\documentclass[12pt, a4paper, leqno]{amsart}
\usepackage[utf8]{inputenc}
\usepackage{amsmath}
\usepackage{amsfonts}
\usepackage{amssymb}
\usepackage{times}
\usepackage[T1]{fontenc} 
\usepackage{url} 
\usepackage{color,esint}
\usepackage[colorlinks, linkcolor=blue, citecolor=blue]{hyperref} 
\usepackage{graphicx} 
\usepackage{enumitem}
\usepackage{tabularx}
    \usepackage{mathtools}
\usepackage{comment}

\let\oldmarginpar\marginpar
\renewcommand\marginpar[1]{\-\oldmarginpar[\raggedleft\footnotesize #1]%
{\raggedright\footnotesize #1}}

\usepackage{amsthm}
\theoremstyle{plain}
\newtheorem{thm}[equation]{Theorem}
\newtheorem{lem}[equation]{Lemma}
\newtheorem{prop}[equation]{Proposition}

\theoremstyle{definition}
\newtheorem{defn}[equation]{Definition}

\theoremstyle{remark}
\newtheorem{rem}[equation]{Remark}

\numberwithin{equation}{section}

\newcommand{\R}{\mathbb{R}}

\newcommand{\Rn}{\mathbb{R}^n}
\def\osc{\operatornamewithlimits{osc}}
\def\essinf{\operatornamewithlimits{ess\,inf}}

\def\supp{\operatornamewithlimits{supp}}

\renewcommand{\div}{\divop}
\renewcommand{\phi}{\varphi}

\def\le{\leqslant}
\def\leq{\leqslant}
\def\ge{\geqslant}
\def\geq{\geqslant}
\def\phi{\varphi}
\def\rho{\varrho}
\def\vartheta{\theta}
\newcommand{\holder}{H\"older }

\def\supp{\operatorname{supp}}
\def\esssup{\operatornamewithlimits{ess\,sup}}
\def\osc{\operatornamewithlimits{osc}}

\def\div{\qopname\relax o{div}}

\def\loc{{\rm loc}}

\newcommand{\ainc}[1]{\hyperref[defn:aInc]{{\normalfont(aInc){\ensuremath{_{#1}}}}}}
\newcommand{\inc}[1]{\hyperref[defn:Inc]{{\normalfont(Inc){\ensuremath{_{#1}}}}}}
\newcommand{\adec}[1]{\hyperref[defn:aDec]{{\normalfont(aDec){\ensuremath{_{#1}}}}}}
\newcommand{\dec}[1]{\hyperref[defn:Dec]{{\normalfont(Dec){\ensuremath{_{#1}}}}}}
\newcommand{\adeci}[1]{\hyperref[defn:aDeci]{{\normalfont(aDec){\ensuremath{_{#1}^\infty}}}}}
\newcommand{\azero}{\hyperref[defn:a0]{{\normalfont(A0)}}}
\newcommand{\aone}{\hyperref[defn:a1]{{\normalfont(A1)}}}

\newcommand{\wvaone}{\hyperref[defn:wva1]{{\normalfont(wVA1)}}}

\date{\today}

\begin{document}

\title[H\"older continuity of the minimizer of an obstacle problem]{H\"older continuity of the minimizer of an obstacle problem with generalized Orlicz growth}
\author{Arttu Karppinen, Mikyoung Lee}
\thanks{A. Karppinen was partially supported by Turku University Foundation. M. Lee was supported by the National Research Foundation of Korea (NRF) grant funded by the Korea Government (NRF-2019R1F1A1061295).}



\begin{abstract}
We prove local $C^{0,\alpha}$- and $C^{1,\alpha}$-regularity for the local solution to an obstacle problem with non-standard growth. These results cover as special cases standard, variable exponent, double phase and Orlicz growth.
\end{abstract}

\maketitle


\section{Introduction}
The classical obstacle problem is motivated by the description of the equilibrium position of an elastic membrane lying above an obstacle. Its mathematical interpretation is to find minimizers of the elastic energy functional with the addition of a constraint that presents the obstacle. This model leads to the mathematical objects called variational inequalities. In this regard, obstacle problem  is deeply related to the study of the calculus of variation and the partial differential equation. It arises in broad applications, such as the study of fluid filtration in porous media, constrained heating, elasto-plasticity, optimal control, and financial mathematics.

The fundamental problem that appears with the study of the obstacle problem is to find the optimal regularity of minimizers. 
The primary model we have in mind is the non-autonomous minimization problem
\begin{equation}\label{minfnal}
 \min_{u} \bigg\{ \int_{\Omega} F(x, \nabla u) \,dx : u \ge \psi \  \textrm{a.e. in } \Omega,  u|_{\partial \Omega} = g \bigg\}
\end{equation}
where $\Omega$ is a bounded open set in $\Rn$ and $\psi$ is the obstacle.
 If $u$ satisfies \eqref{minfnal}, we say that $u$ is a solution to the obstacle problem.  
 In this paper, we are interested in H\"older regularity properties
 of solutions to the obstacle problems related to \eqref{minfnal} with nonstandard growth conditions. 
  In fact, the regularity results for such obstacle problems have been achieved by applying results and  techniques developed in the research on the unconstrained case, i.e. when $\psi = -\infty$.

For the unconstrained case, there have been many extensive researches on the regularity theory, including $C^{0,\alpha}$- and $C^{1,\alpha}$-regularity, for the nonstandard growth problem such as the non-autonomous minimization problem \eqref{minfnal} when $F$ satisfies $(p,q)$-growth conditions, that is,
$ |z|^p \lesssim F(x, z) \lesssim |z|^q +1,  \ p < q$, 
starting with Marcellini’s seminal papers \cite{Mar89, Mar91}.
Several model functionals mainly in relation to the Lavrentiev phenomenon were proposed by Zhikov \cite{Zhi95} in order to describe the behavior of anisotropic materials in the framework of homogenization and nonlinear elasticity. The main feature of such functionals, including 
the variable exponent functionals with
$$F(x, z) \approx |z|^{p(x)} \,  \ \textrm{for } 1<\inf_{\Omega} p(x) \le \sup_{\Omega} p(x) < \infty,$$ 
and the double phase functionals with
$$F(x, z) \approx |z|^{p}+ a(x)|z|^{q} \, \ \textrm{for } 1< p \le q < \infty \text{ and } a(\cdot) \ge 0,$$ 
 is that those integrands $F(x, z)$ change their ellipticity and growth properties according to the point $x$. 
 Starting with a higher integrability result of Zhikov \cite{Zhi97}, 
regularity problems for minimizers of  the variable exponent functionals have been vigorously studied in for instance \cite{AceM01,CosM99, Ele04,Fan07} (see also references in the survey \cite{Min06}).  
For the double phase functionals, the regularity theory was developed by Baroni, Colombo and Mingione in a series of remarkable papers \cite{BarCM16,BarCM18,ColM15a}. We also refer to \cite{CupGGP18, DefO19,GiaP13,MinS99, Ok18,RagT20} for regularity results in the various variants and borderline cases. 
Maximal regularity for Orlicz growth was settled by Lieberman in \cite{Lie91} and for H\"older continuity of the solution assumptions have been relaxed in \cite{ArrH18}. Other regularity results for Orlicz growth can be found for example in \cite{Cia97,CiaF99,DieSV09}.
 Furthermore, regularity properties have been recently studied for minimizers of the functionals with the generalized Orlicz growth that cover the functionals referred above, see for instance \cite{HarHK17,HarHM18,HarHLT13}. 

Our first main result, Theorem \ref{thm:C^alpha}, concerns local $C^{0,\alpha}$-continuity of the solution for some $\alpha \in (0,1)$. This standard H\"older continuity of solution to an obstacle problem is relevant by the side of maximal regularity results since less assumptions are needed for the functional and the obstacle. We do not require differentiability of the functional and the obstacle is assumed to be merely H\"older continuous rather than having continuous gradient. These lighter results are applied for example in study of removable sets \cite{ChlF20,KilZ02}.

In the very recent paper \cite{HasO_pp18},
$C^{0,\alpha}$- and  $C^{1,\alpha}$- regularity properties were established by H\"ast\"o-Ok for minimizers of non-autonomous functionals with sharp and general conditions, covering all the previous results.  It is worth pointing out that this result assumed no gap between growth exponents $p$ and $q$. Instead their work is based on a carefully crafted continuity assumption \wvaone{} for the $x$ variable in $\phi$. This assumption turns out to capture sharp structural assumptions, which ensure regularity, in the important special cases such as double phase and variable exponent cases.
 Following their results and ideas, we are concerned with the regularity properties for solutions of the obstacle problems under $(p,q)$-growth condition. These results are collected in Theorem \ref{thm:C^1,alpha}.

 For the obstacle problem, it is noted that the solution inherits the regularity properties from the obstacle. In particular, for the linear obstacle problem, i.e., when $F(x, z) \approx |z|^2$ in \eqref{minfnal}, it is well known that the solution has the same regularity as the obstacle $\psi$, see for instance \cite{BreK73,CafK80}. However this is not usually permitted in the nonlinear cases. Hence extensive research in this direction has been done into the regularity of solutions to the nonlinear obstacle problems.  As the first result for the nonlinear functional with standard growth, i.e. when $F(x, z) \approx |z|^p$ with $1<p<\infty$, Michael-Ziemer \cite{MicZ86} proved that the solution is H\"older continuous when the obstacle $\psi$ is H\"older continuous.  
Choe \cite{Cho91} established the same result for the gradient of solutions when the gradient of the obstacle $\psi$ is H\"older continuous.
 We further refer to \cite{ChoL91,Fuc90,FucM00,MuZ91,Ok17} for the H\"older regularity results on the nonlinear obstacle problems related to $p$-Laplace type functionals and more general functionals. 
Concerning with the nonstandard growth cases, H\"older type regularity results for obstacle problems with $p(x)$-growth were obtained in several papers for instance \cite{EleH08,EleH11,EleHL13}. The $L^{p(\cdot)} \log L$-growth case, which can be regarded as a borderline case lying between Orlicz growth and variable exponent growth, was considered in the recent paper \cite{Ok16}. H\"older regularity results for the double phase case have been studied for example in \cite{ChlF20}.

Obstacle problems for functionals having $(p,q)$ growth have been studied in for example \cite{CasEM_pp19,Def19} with different assumptions. We do not assume any gap between exponents $p$ and $q$ or $C^2$-regularity of our functional, but instead work with the \wvaone{} condition.
Results of this paper cover all the previous results in the nonstandard growth cases mentioned above. We would like to point out that our results are new in many borderline cases including the double phase case.





\section{Preliminaries and main results}


By $\Omega \subset \Rn$ we denote a bounded domain, i.e.\ an open and connected. If for an open set $\Omega'$ we have $\overline{\Omega'} \subset \Omega$, we denote it simply as $\Omega' \Subset \Omega$.
By $Q_r$ we mean a cube with side length $2r$ and by $B_r$ a ball with radius $r$. For an integrable function $g : U\subset \mathbb{R}^n \rightarrow \mathbb{R}$, we define the average of $g$ in $U$ by 
$$ \bar{g}_U : = \fint_{U} g \;dx = \frac{1}{|U|}\int_{U} g \;dx.$$
\subsection*{Generalized $\Phi$-functions}
In this section we introduce the basic notations, definitions and assumptions for our growth rate. 

A function $f$ is said to be  \textit{almost increasing} if there
exists $L \ge 1$ such that $f(s) \le L f(t)$ for all $s \le t$. 
\textit{Almost decreasing} is defined analogously.
By \textit{increasing} we mean that the inequality holds for $L=1$ 
(some call this non-decreasing), similarly for \textit{decreasing}.

\begin{defn}
We say that $\phi: \Omega\times [0,\infty) \to [0,\infty]$ is a 
\textit{ (generalized) $\Phi$-prefunction} if the following hold:
\begin{itemize}
\item[(i)]  The function $x \mapsto \phi(x,t)$ is 
measurable for every $t \in [0,\infty)$.
\item[(ii)] The function $t \mapsto \phi(x,t)$ is 
non-decreasing for every $x \in \Omega$.
\item[(iii)] $\lim\limits_{t \to 0^+} \phi(x,t) = \phi(x,0)=0$ and $
\lim\limits_{t \to\infty} \phi(x,t)= \infty$ for every $x \in \Omega$.
\end{itemize}
A $\Phi$-prefunction $\phi$ is called a \textit{(generalized weak)
$\Phi$-function}, denoted by $\phi \in \Phi_w(\Omega)$, if the function $t \mapsto \frac{\phi(x,t)}{t}$ is almost 
increasing in $(0, \infty)$ for every $x\in \Omega$,
and  \textit{a (generalized) convex
$\Phi$-function}, denoted by $\phi \in \Phi_c(\Omega)$, if 
the function $t \mapsto \phi(x,t)$ is 
left-continuous and convex for every $x\in \Omega$. 
   Additionally, we denote $\phi \in \Phi_w$ or $\phi \in \Phi_c$ if $\phi$ is independent of the space variable $x$.
\end{defn}

If $\phi \in \Phi_c(\Omega)$, we note that there exists its right-derivative $\phi' = \phi'(x,t)$, which is non-decreasing and right continuous, 
satisfying 
$$ \phi(x,t)= \int_{0}^t \phi'(x,s)\;ds.$$ This derivative is also denoted by $\phi_t$.

Now let us consider $\phi \in \Phi_w(\Omega)$ and $\gamma>0$. By $\phi^{-1}(x, \cdot) : [0,\infty) \rightarrow [0,\infty]$ we denote the left-continuous inverse of $\phi$ defined by 
\begin{equation*}
    \phi^{-1}(x, t) := \inf \{\tau \geq 0 : \phi(x,\tau) \geq t \}.
\end{equation*}
We define  some conditions on $\phi$ which are related to regularity properties with respect to the $x$-variable and the $t$-variable. We say that $\phi$ satisfies 
\begin{itemize}
\item[(A0)]\label{defn:a0}
if there exists $L \geq 1 $ such that $ \frac1L \le \phi^{-1}(x,
1) \le L$ for every  $x \in \Omega$.
\item[(A1)]\label{defn:a1}
if there exists $L \geq 1 $ such that for any ball $B_r \Subset \Omega$ with $ |B_r| <1,$
\[
\phi^{+}_{B_r}(t) \le L \phi^{-}_{B_r} (t) 
\quad\text{for all}\quad t >0  \quad\text{with}\quad
\phi^{-}_{B_r}(t) \in \bigg[1, \frac{1}{|B_r|}\bigg].
\]

\item[(wVA1)] \label{defn:wva1}
if for any $\epsilon>0$, there exists a non-decreasing continuous function $\omega=\omega_\epsilon:[0,\infty) \to [0,1]$ with $\omega(0)=0$ such that for any small ball $B_r \Subset \Omega,$
\begin{align*}
\phi^+_{B_r}(t) \leq (1+\omega(r))\phi^{-}_{B_r}(t)+\omega(r) \quad\text{for all}\quad t >0  \quad\text{with}\quad \phi^-_{B_r}(t) \in \bigg[
\omega(r), \frac{1}{|B_r|^{1-\epsilon}}\bigg].
\end{align*}

\item[(aInc)$_{\gamma}$] \label{defn:aInc} if
$t \mapsto \frac{\phi(x,t)}{t^{\gamma}}$ is almost 
increasing 
with constant $L \geq 1$ uniformly in  $x\in\Omega$.
\item[(Inc)$_{\gamma}$] \label{defn:Inc} if
$t \mapsto \frac{\phi(x,t)}{t^{\gamma}}$ is non-decreasing for every \ $x\in\Omega$.
\item[(aDec)$_{\gamma}$] \label{defn:aDec}
if
$t \mapsto \frac{\phi(x,t)}{t^{\gamma}}$ is almost 
decreasing with constant $L\ge 1$ uniformly in $x\in\Omega$.
\item[(Dec)$_{\gamma}$] \label{defn:Dec} if
$t \mapsto \frac{\phi(x,t)}{t^{\gamma}}$ is non-increasing for every \ $x\in\Omega$.
\end{itemize} 
We would like to explain these assumptions more informally. Firstly, \azero{} condition places us in an "unweighted" space, that is, $\phi$ is not singular or degenerate with respect to the spatial variable. Secondly, conditions \aone{} and \wvaone{} are regularity conditions with respect to the spatial variable. The former one is a jump condition, where as the latter is a refined continuity condition. In many cases $\varepsilon$ could be equal to 0, and this would imply regularity for many special cases such as variable exponent and double phase. However, the weak form catches interesting borderline assumptions for example in double phase case. Last four conditions control the growth of the $\Phi$-function with respect to the $t$-variable. Often we want $\gamma$ in \ainc{\gamma} to be strictly greater than 1 to exclude $L^1$ case and finite in \adec{\gamma} to exclude $L^\infty$ case. The ''almost'' part is more flexible and is invariant under equivalent $\Phi$-functions (see the next paragraph). However it allows for local exceptions in growth rate, which \inc{\gamma} and \dec{\gamma} exclude.
 We also note that \adec{\gamma} implies doubling $\phi(x,2t) \leq c\, \phi(x,t)$.

 The notation $f_1\lesssim f_2$ means that there exists a constant
$C>0$ such that $f_1\le C f_2$. The notation $f_1\approx f_2$ means that
$f_1\lesssim f_2 \lesssim f_1$ whereas $f_1\simeq f_2$ means that 
$f_1(t/C)\le f_2(t)\le f_1(Ct)$ for some constant $C\ge 1$. Throughout this paper, we use these notations when the relevant constants $C$ depend on  $n$ and constants in our conditions such as \ainc{\gamma}, \inc{\gamma}, \adec{\gamma}, \dec{\gamma}, and \azero.

\subsection*{Generalized Orlicz space}
Let $\phi \in \Phi_w(\Omega)$ and $L^0(\Omega)$ be the set of the measurable functions on $\Omega$. The \emph{generalized Orlicz space} is defined as the 
set 
$$L^{\phi}(\Omega) : = \bigg\{ f \in L^0(\Omega) : \| f \|_{L^{\phi}(\Omega)} \leq \infty \bigg\}
$$
 with the (Luxemburg) norm
$$ \| f \|_{L^{\phi}(\Omega)} := \inf \bigg\{ \lambda >0 : \varrho_{L^{\phi}} 
\Big( \frac{ f}{\lambda}\Big) \le1\bigg\},$$
where $ \varrho_{\phi}(g)$ is the modular of $g \in L^0(\Omega)$ defined by 
\[
\varrho_{L^{\phi}}(g):=\int_{\Omega} \phi(x, |g(x)|)\,dx. 
\]

A function $f \in L^{\phi}(\Omega)$ belongs to the \emph{Orlicz--Sobolev space 
$W^{1,\phi}(\Omega)$} if its weak partial derivatives $\partial_1 f, \dots 
\partial_n f$ exist and belong to $L^{\phi}(\Omega)$ with the norm 
\[
\| f \|_{W^{1,\phi}(\Omega)} := \| f \|_{L^{\phi}(\Omega)} + \sum_{i} \| \partial_{i}f \|_{L^{\phi}(\Omega)}.
\]
Furthermore,  we denote by
$W^{1,\phi}_0 (\Omega)$ the closure of $C^{\infty}_0(\Omega)$ in 
$W^{1,\phi}(\Omega)$.

\subsection*{Main results} Let $\Omega$ be a bounded open set in $\Rn$. 
Recall that a function $u \in W^{1,\phi}_{\loc}(\Omega)$ is called
 \textit{a (local)
minimizer} of the functional
\begin{equation}\label{mainfnal}
 \int_{\Omega} \phi(x, |\nabla u|) \,dx
\end{equation}
if  for 
every open $\Omega' \Subset \Omega$ and every $v \in W^{1,\phi}(\Omega')$ 
with $u-v\in W^{1,\phi}_0(\Omega')$ we have
\[
\int_{\Omega'}\phi(x, |\nabla u|)\,dx \le  \int_{\Omega'} \phi(x,|\nabla v|)\,dx.
\] 
For a function $\psi: \Omega \rightarrow [-\infty,\infty)$ called the \textit{obstacle}, the class of admissible
functions is defined by
 \begin{align*}
\mathcal{K}^{\phi}_{\psi}(\Omega) := \{ u \in W^{1,\phi}(\Omega) \, | \, u \geq \psi \text{ a.e. in } \Omega\}. 
\end{align*} 
We define the related solution of the obstacle problem as follows.

\begin{defn}
We say that a function $u \in \mathcal{K}^{\phi}_\psi(\Omega)$ is \textit{a  (local) minimizer} of  \eqref{mainfnal} in $\mathcal{K}^{\phi}_\psi(\Omega)$  if it satisfies
\begin{align*}
\int_{\Omega'} \phi(x,|\nabla u|) \, dx \leq \int_{\Omega'} \phi(x,|\nabla w|) \, dx,
\end{align*}
where $w \in \mathcal{K}^{\phi}_{\psi}(\Omega)$ 
and $\Omega' \Subset \Omega$.
If $\phi \in \Phi_w(\Omega) \cap C^1([0,\infty))$, we say that a function $u \in \mathcal{K}^{\phi}_{\psi}(\Omega)$ is a solution to the $\mathcal{K}^{\phi}_{\psi}(\Omega)$-obstacle problem 
 if it satisfies 
 $$ \int_{\Omega}\left(  \frac{ \partial_t\phi(x,|\nabla u|)}{|\nabla u|} \nabla u \right) \cdot \nabla(\eta - u) \,dx \geq 0$$
 for all $\eta \in  \mathcal{K}^{\phi}_{\psi} (\Omega)$ with $\mathrm{supp}(\eta-u) \subset \Omega,$ which is equivalent to 
 \begin{equation}\label{weakformobstacle}
  \int_{\Omega}\left(  \frac{ \partial_t\phi(x,|\nabla u|)}{|\nabla u|} \nabla u \right) \cdot \nabla\eta  \,dx \geq 0
  \end{equation}
 for all $\eta \in W^{1,\phi}(\Omega)$ with a compact support and $\eta \geq \psi - u $ a.e. in $\Omega.$
\end{defn}

\begin{rem}
 We note that a solution is also a minimizer and the reverse implication requires differentiability for the $\Phi$-function. 
More precisely, in the case where $\phi(x,t)$ is differentiable with respect to $t$ and satisfies \adec{}, then solutions and minimizers  are equivalent, see \cite[Theorem 7.6]{HarHK16}.
\end{rem}

 For simplicity,  let us define $\partial \phi = \partial \phi(x,t) : \Omega \times \Rn \rightarrow \Rn $ by 
 $$ \partial \phi(x, t) 
 = \frac{\partial_t \phi(x,|t|)} {|t|} t .$$
We note that for $g \in W^{1,\phi}(\Omega)$ with $ g \geq \psi$ on $\partial \Omega$, the minimizer of 
the functional 
$$ u \in \{ w \in \mathcal{K}^{\phi}_{\psi}(\Omega) :  w = g \text{ on } \partial \Omega\} \mapsto  \int_{\Omega} \phi (x, |\nabla u|)\,dx $$
is the solution to the obstacle problem of $\mathcal{K}^{\phi}_{\psi}(\Omega)$ with $u = g $ on $\partial \Omega.$

Our first result concerns the local H\"older continuity of the solution. This result does not assume differentiability of our $\Phi$-function or \wvaone{}. Instead, \aone{} condition is enough combined with local H\"older continuity of the obstacle function. 
\begin{thm}
\label{thm:C^alpha} 
Let $\Omega \subset \mathbb{R}^n$ be a bounded domain and 
 $\phi\in\Phi_w(\Omega)$ 
satisfy \ainc{}, \adec{},  \azero{}, and \aone{}.  
Let $u$ be a minimizer of \eqref{mainfnal} in $\mathcal{K}^{\phi}_{\psi}(\Omega).$
Suppose that the obstacle $\psi \in C_{\loc}^{0,\beta}(\Omega)$ for some $\beta\in(0,1).$ Then 
$u \in C_{\loc}^{0,\alpha}(\Omega)$ for some $\alpha \in (0,1).$
\end{thm}
The proof is based on constructing classical Harnack's inequality. First, a supremum estimate of the solution is proved via a use of Caccioppoli type energy estimate and results in \cite{HarHM18}. Compared to previous work, we need to take care of an integral average term to match the supremum estimate with the corresponding infimum estimate. 

The following second theorem studies maximal regularity of the solution. Here we need to strengthen our assumptions to include $C^1$-regularity of $\phi$ and the obstacle $\psi$ and replace the assumption \aone{} with \wvaone{}. $C^{1,\alpha}$-regularity requires also a standard decay estimate for modulus of continuity.

\begin{thm}\label{thm:C^1,alpha}
Let $\Omega \subset \mathbb{R}^n$ be a bounded domain, $\phi\in\Phi_w(\Omega)$
 and  $\phi(x, \cdot) \in C^{1}([0,\infty))$ for any $x \in \Omega$ with $\partial_t \phi $ satisfying    \azero{}, \inc{p-1}, \dec{q-1} for some $1<p\leq q$. Let 
 $u \in \mathcal{K}^{\phi}_{\psi}(\Omega)$ be a solution to the $\mathcal{K}^{\phi}_{\psi} (\Omega)$-obstacle problem and suppose $\psi \in C^{1,\beta}_{\loc}(\Omega)$ for some $\beta \in (0,1)$.

\begin{itemize}
\item[(i)]  If $\phi$ satisfies  \wvaone{},  then  $u \in C_{\loc}^{0,\alpha}(\Omega)$ for any $\alpha \in (0,1).$
\item[(ii)]  If $\phi$ satisfies \wvaone{} with
 $$\omega(r) \lesssim r^{\delta} \text{ for all $r \in (0,1]$  and for some }  \delta >0,$$
 then $u \in C_{\loc}^{1,\alpha}(\Omega) $ for some  $\alpha \in (0,1).$
\end{itemize}
\end{thm}
For higher regularity, the modulus of continuity with respect to $x$ is expected to have vanishing property already in the case of Laplace equation. In addition, continuity of the gradient also requires power-type decay estimate for classical equations, see  \cite{GiaG_83}.

For proving results in Theorem~\ref{thm:C^1,alpha}, we first obtain 
the higher integrability of the gradient of solutions to the obstacle problem which implies the reverse H\"older type inequality. Taking into account the regularized Orlicz function $\tilde{\phi}$, which was constructed by \cite{HasO_pp18}, we derive comparison estimates for the gradients of solutions to the $\mathcal{K}^{\tilde \phi}_{\psi}(\Omega)$-obstacle problem and to $\tilde \phi$-Laplacian equations. The proofs conclude with classical iteration arguments.

\subsection*{Inequalities for generalized Orlicz functions}

 Below introduce some pointwise inequalities for generalized Orlicz functions. 
In generalized Orlicz spaces, the classical \holder conjugate exponent is replaced by a conjugate $\Phi$-function
\begin{align}
\label{defn:conjugate}
\phi^\ast (x,t) := \sup_{s\geq 0} \big(ts - \phi(x,s) \big).
\end{align} 
From the definition it is immediate that the generalized Young's inequality 
\begin{align}
\label{eq:gen-young}
st \leq \phi(x,s) + \phi^\ast(x,t)
\end{align}
holds for all $s,t \geq 0$.

The following lemmas deal with derivatives of $\Phi$-functions and their proofs can be found in \cite[Proposition 3.8]{HasO_pp18}. 
\begin{lem}
\label{est:Orlicz}
Let $\phi \in \Phi_c \cap C^{1}([0,\infty))$ with $\phi'$ satisfying \inc{p-1} and \dec{q-1} for some $1<p\leq q$. Then for $ \kappa \in (0,\infty)$ and $x,y \in \R^n,$ the following are satisfied:
\begin{enumerate}
\item{$t \phi'(t) \approx \phi(t)$ and $\phi$ satisfies \inc{p} and \dec{q};}
\medskip
\item{$\dfrac{\phi'(|x|+|y|)}{|x|+|y|}|x-y|^2 \approx \left ( \dfrac{\phi'(|x|)}{|x|}x - \dfrac{\phi'(|y|)}{|y|}y \right ) \cdot (x-y)$;}
\medskip
\item{$\dfrac{\phi'(|x|+|y|)}{|x|+|y|}|x-y|^2 \lesssim \phi(|x|) - \phi(|y|) - \dfrac{\phi'(|y|)}{|y|} y \cdot (x-y)$;}
\medskip
\item{$\phi(|x-y|) \lesssim \kappa \left[ \phi(|x|) + \phi(|y|) \right]+ \kappa^{-1} \dfrac{\phi'(|x|+|y|)}{|x|+|y|}|x-y|^2$.}
\end{enumerate}
Moreover, if $\phi \in C^2((0,\infty))$, then $t \phi''(t) \approx \phi'(t)$ and $\frac{\phi'(|x|+|y|)}{|x|+|y|}$ can be replaced by $\phi''(|x|+|y|)$ in (2)--(4).
\end{lem}

\begin{lem}[Propositions 3.5 \& 3.6 \cite{HasO_pp18}]
\label{lem:phi-estimates}
Let $\phi$ be a $\Phi$-prefunction.
\begin{enumerate}
\item{If $\phi$ satisfies \ainc{1}, then there exists $\eta \in \Phi_c(\Omega)$ such that $\phi \simeq \eta$}.
\medskip
\item{If $\phi$ satisfies \adec{1}, then there exists $\eta \in \Phi_c(\Omega)$ such that $\phi \approx \eta^{-1}$. Note that $\eta^{-1}(x, \cdot)$ is concave.}
\medskip
\item{Let $p,q \in (1, \infty)$. Then $\phi$ satisfies 
\ainc{p} or \adec{q} if and only if $\phi^\ast$ satisfies \adec{\frac{p}{p-1}} or \ainc{\frac{q}{q-1}}, respectively.}

\medskip
\item{If $\phi$ satisfies \ainc p and \adec q,  then for any $s,t \geq 0$ and $\kappa \in (0,1)$,
\begin{align*}
ts \leq \phi(x, \kappa^{1/p}t) + \phi^\ast(x, \kappa^{-1/p}s) \lesssim \kappa \phi(x,t) + \kappa^{\frac{-1}{p-1}}\phi^\ast(x,s)
\end{align*} and 
\begin{align*}
ts \leq \phi(x, \kappa^{1/q'}t) + \phi^\ast(x, \kappa^{-1/q'}s) \lesssim  \kappa^{-(q-1)}\phi(x,t) +\kappa \phi^\ast(x,s).
\end{align*} 
}

\item{If  $\phi\in \Phi_c(\Omega)$ and $\phi'$ satisfies  \adec{q} then 
\begin{align*}
\phi^\ast \left (x, \dfrac{\phi(x,t)}{t}\right ) \leq \phi^\ast (x, \phi'(x,t)) \leq t \phi'(t) \approx \phi(t).
\end{align*}
}
\end{enumerate}
\end{lem}

Lastly we have a Sobolev--Poincar\'e inequality which is used for higher integrability results.

\begin{lem}[Sobolev--Poincar\'e inequality, Theorem 3.3 \cite{HarHM18}] \label{thm:SP}
Let $Q_r \subset \Rn$ be a cube or a ball with diameter $2r$.
Let $\phi \in \Phi_w(Q_r)$ satisfy \ainc{p}, \adec{q},  \azero{}, and \aone{}. 
Let $s < \frac{n}{n-1}.$
Then there exists $\beta = \beta(n,s,\phi)> 0$ such that
\begin{equation}\label{SPineq2}
\bigg(\fint_{Q_r} \phi \bigg(x, \frac{\beta\,|u|}r\bigg)^s\,dx 
\bigg)^{\frac1s}
\lesssim
\fint_{Q_r} \phi (x, |\nabla u|) \,dx  + \frac{|\{\nabla u \neq 0\}\cap {Q_r}|}
{|{Q_r}|}
\end{equation}
for any $u \in W^{1,1}_0({Q_r})$ with $\| \nabla u \|_{L^{\phi}} <1$. 
If additionally $s\le p$, then 
\begin{equation}\label{SPineq}
\fint_{Q_r} \phi \bigg(x, \beta\,\frac{|u-u_{Q_r}|}{r}\bigg)\,dx 
\lesssim
\bigg(\fint_{Q_r} \phi (x, |\nabla u|)^{\frac1s} \,dx \bigg)^{s} + 1
\end{equation}
for any $u \in W^{1,1}({Q_r})$ with $\| \nabla u \|_{L^{\phi^{1/s}}}\le M$, and the implicit constant depends on $M$.
The average $u_{Q_r}$ can be replaced by $u_Q$ for some cube or ball $Q\subset Q_r$ 
with $|Q| > \mu |Q_r|$, in which case the constant depends also on $\mu$.  
\end{lem}

\subsection*{Essential supremum and infimum estimates}
\noindent Our first main theorem proven in Section \ref{sec:thm1} is local H\"older continuity of a solution to an obstacle problem with H\"older continuous obstacle $\psi$. This follows the classical route of estimating supremum and infimum of minimizer $u$ with its integral averages. 
We prove the essential supremum result in the context of solutions to highlight that the result follows also in that case. The only difference this makes is in the proof of Caccioppoli type inequality in Proposition \ref{prop:sup-estimate} and the requirement for the $\Phi$-function to be differentiable with respect to the second variable. For minimizers, the proof can be modified from an analogous result in \cite{HarHM18}.

For the following we write
 \[
 A(k,r) := Q_r \cap \{u >k\}. 
\]

We first recall the supremum 
 bounds for the local minimizer $u$ of the $\phi$-energy.

\begin{prop}[Proposition 5.5 and Corollary 5.8 \cite{HarHM18}]
\label{prop:esssupinf}Let $\phi\in\Phi_w(\Omega)$ 
satisfy \ainc{p}, \adec{q},  \azero{}, and \aone{}. 
Suppose that $u\in W^{1,\phi}_\loc(\Omega)$
satisfies $\rho_{L^\phi(Q_{2r})}(\nabla u)\le 1$ for  $ Q_{2r}\subset \Omega$. 
Suppose that $u$ satisfies the Caccioppoli inequality 
\begin{equation}\label{eq:caccioppoli}
\int_{A(\ell, r)} \phi(x,|\nabla (u-\ell)_+|) \,dx 
\lesssim 
\int_{A(\ell, 2r)} \phi\Big(x,\frac{(u-\ell)_+}{r}\Big)  \,dx
\end{equation}
for any $\ell \geq 0$.
 Then 
$u_+$ is bounded and 
\begin{equation}\label{eq:supbdd}
\esssup_{Q_{r/2}} u_+
 \lesssim  
    \bigg( \fint_{Q_r}u_+^q\,dx\bigg)^{\frac1q} 
		+ |u_{Q_{r/2}}| +r
\end{equation}
for any $Q_{r}\subset \Omega$.
The term $|u_{Q_{r/2}}|$ can be omitted if $u$ is non-negative. 

Furthermore, 
if $u \in L^{\infty}(Q_r)$ satisfies \eqref{eq:supbdd} without the term $|u_{Q_{r/2}}|$, then 
\begin{equation}\label{eq:supbddh}
\esssup_{Q_{r/2}} u_+
 \lesssim  
    \bigg( \fint_{Q_r} u_+^h\,dx\bigg)^{\frac1h} 
		  +r
\end{equation}
for any $h \in (0,\infty).$
\end{prop}

%

We need similar  supremum and infimum bounds for the solutions of the obstacle problem. We start with the supremum estimate and base our proof on \cite[Section 5]{HarHM18}. Therefore we only need to prove that solution to the obstacle problem satisfies the Caccioppoli type energy estimate \eqref{eq:caccioppoli}. Note that in the case of obstacle problems, we need to restrict possible values of $\ell$ using the obstacle $\psi$.

\begin{prop}
\label{prop:sup-estimate}
Let $\phi\in\Phi_w(\Omega) \cap C^1([0,\infty))$ satisfy \ainc{p}, \adec{q},  \azero{}, and \aone{}. 
Let $u \in \mathcal{K}^{\phi}_{\psi}(\Omega)$ be a solution to the $\mathcal{K}^{\phi}_{\psi}(\Omega)$- obstacle problem. Then if $\psi \in W^{1,\phi}(\Omega) \cap L^{\infty}_{\loc}(\Omega)$ and $\theta \in \left [ \frac{1}{2}, 1 \right )$ we have
\begin{align}\label{eq:supbddobstacle}
\esssup_{Q_{\theta r}} (u-\ell)_+ \lesssim (1-\theta)^{-4nq^2} \left[   \left(\fint_{Q_{r}} (u-\ell)_+^{q} \, dx \right )^{1/q} +| (u-\ell)_{Q_{r/2}} | \right ] + r 
\end{align}
for  any $Q_{2r}\subset \Omega$
and $\ell \geq \sup_{Q_{2r}} \psi,$ provided that $\rho_{L^\phi(Q_{2r})}(|\nabla u|) \le 1$. The term $|(u-\ell)_{Q_{r/2}}|$ can be omitted if $u-\ell$ is non-negative almost everywhere in $Q_{2r}$.

Furthermore, 
if $u \in L^{\infty}(Q_r)$ satisfies \eqref{eq:supbddobstacle} without the term $| (u-\ell)_{Q_{r/2}} | $, then 
\begin{equation}\label{eq:supbddhobstacleh}
\esssup_{Q_{r/2}} (u-\ell)_+
 \lesssim  
    \bigg( \fint_{Q_r} (u-\ell)_+^h\,dx\bigg)^{\frac1h} 
		  +r
\end{equation}
for any $h \in (0,\infty).$

\end{prop}

\begin{proof}
If $v:= u- \ell$ satisfies Caccioppoli inequality \eqref{eq:caccioppoli}, then Proposition~\ref{prop:esssupinf} implies the desired bound.
Hence it suffices to show that $v$ satisfies \eqref{eq:caccioppoli} for  $\ell \geq \sup_{Q_{2r}} \psi.$

Consider any cubes $ Q_{\sigma} \subset Q_{\rho} \subset Q_{2r}$. Let $k \geq 0$ and let $\tau \in C^{\infty}_0 (Q_\rho)$ be a cut off function such that $0 \leq \tau \leq 1$ in $Q_\rho,  \ \tau = 1$  in $Q_{\sigma}$ and $|\nabla \tau| \leq \frac{c(n)}{\rho-\sigma}$.

Since $(v-k)_+ = (u-\ell -k)_+ \leq u -\psi$, we take a test function $\eta= -(v-k)_+ \tau^q$ in \eqref{weakformobstacle} to discover that 
\begin{align*}
& \int_{A(k+\ell, \rho)}\partial \phi(x, |\nabla u|) \cdot \nabla\left(- (v-k)_+ \tau^q \right) \,dx  \\
 &  = -\int_{A(k+\ell, \rho)}\left[  \partial \phi(x, |\nabla u|)  \cdot \nabla u\right]\tau^q  \,dx - q\int_{A(k+\ell, \rho)}\left[ \partial \phi(x, |\nabla u|) \cdot \nabla \tau \right] (v-k)_+  \tau^{q-1} \,dx \geq 0,
\end{align*} 
where $A(k+\ell, \rho) := Q_{\rho} \cap \{ u > \ell + k \}.$
Then from Lemma~\ref{lem:phi-estimates} (5) and the fact that $\phi^*(x,\phi_t(x,t)) \approx \phi^*(x, \phi(x,t)/t) \approx \phi(x,t)$ and $\phi^*$ satisfies \ainc{\frac{q}{q-1}}, we deduce with \adec{} that 
\begin{align*}
& \int_{A(k+\ell, \rho)} \phi(x, |\nabla u|) \tau^q \,dx \lesssim \int_{A(k+\ell, \rho)} \partial \phi(x, |\nabla u|) |\nabla u| \tau^q \,dx \\
& \quad \leq q\int_{A(k+\ell, \rho)}|\partial \phi(x, |\nabla u|)
| |\nabla \tau|  (v-k)_+  \tau^{q-1} \,dx \\
& \quad\lesssim \kappa \int_{A(k+\ell, \rho)} \phi^*(x, |\partial\phi(x,|\nabla u|)| \tau^{q-1} )\,dx  + \frac{1}{\kappa^{q-1}} \int_{Q_{\rho}} \phi(x, |\nabla \tau| (v-k)_+ )\,dx   \\
&  \quad\lesssim \kappa \int_{A(k+\ell, \rho)} \phi^*(x, |\partial\phi(x,|\nabla u|)|) \tau^q\,dx  + \frac{1}{\kappa^{q-1}}\int_{Q_{\rho}} \phi(x, |\nabla \tau| (v-k)_+ )\,dx   \\
&   \quad\lesssim \kappa \int_{A(k+\ell, \rho)} \phi(x, |\nabla u|) \tau^q\,dx  + c_{\kappa} \int_{Q_{\rho}} \phi\left(x, \frac{(v-k)_+}{\rho-\sigma} \right)\,dx
\end{align*}
for any $\kappa \in (0,1).$ 
Therefore, by choosing $\kappa$ sufficiently small, we conclude 
\begin{equation*}
 \int_{Q_{\sigma}} \phi(x, |\nabla (v-k)_+|) \tau^q \,dx \lesssim  \int_{Q_{\rho}} \phi\left(x, \frac{(v-k)_+}{\rho-\sigma} \right)\,dx. \qedhere
\end{equation*}
\end{proof}

\medskip
Since solution to an obstacle problem is also a superminimizer, the standard arguments provide the following weak Harnack inequality, see \cite[Corollary 6.4]{HarHM18}.

\begin{prop}
\label{prop:inf-estimate}
Let $\phi\in\Phi_w(\Omega)$ 
satisfy \ainc{p}, \adec{q},  \azero{}, and \aone{}. 
Suppose that  $u\in W^{1,\phi}_\loc(\Omega)$
is a non-negative solution of the $\mathcal{K}^{\phi}_{\psi}(\Omega)$-obstacle problem.
Then there exists $h_0>0$ such that
\begin{align*}
 \left( \fint_{Q_r} u^{h_0}\,dx\right)^\frac1{h_0}
 \lesssim
 \essinf_{Q_{r/2}} u +  r
\end{align*}
when $Q_{2r}\subset \Omega$ and 
$\rho_{L^\phi(Q_{2r})}(|\nabla u|) \le 1$.
\end{prop}


\section{H\"older continuity}\label{sec:thm1}

Now we are ready to prove local H\"older continuity of the solution to an obstacle problem with H\"older continuous obstacle. Compared to higher regularity results in later sections, we assume the weaker condition \aone{} instead of \wvaone{}.

\begin{proof}[Proof of Theorem~\ref{thm:C^alpha}]
Since $\psi \in C_{\loc}^{0,\beta}(\Omega)$ for some $\beta \in (0,1),$ we note that for any $ Q_{2r} \subset \Omega$ with $r<1$, there exists a constant $[\psi]_\beta>0$ such that 
\begin{align*}
|\psi(x)-\psi(y)| \leq [\psi]_{\beta} |x-y|^\beta \ \text{ for all  }x,y \in Q_{2r}.
\end{align*}
From Propositions \ref{prop:sup-estimate}-\ref{prop:inf-estimate}, we note that $u$ is locally bounded in $\Omega$. Then for $Q_{2r}\subset \Omega$,
%
 we define
\begin{align*}
\overline{u}(r) := \esssup_{x \in Q_r} u(x), \ \  \underline{u}(r):= \essinf_{x \in Q_r} u(x), \ \ \overline{\psi}(r) := \esssup_{x \in Q_r} \psi(x), \ \ \underline{\psi}(r) := \essinf_{x\in Q_r} \psi(x).
\end{align*}

Next we consider two cases: $ \underline{u}(r) \geq \overline{\psi}(r) $ and $\underline{u}(r) \leq \overline{\psi}(r)$. For the first case we can use Proposition \ref{prop:sup-estimate} with $\ell=\underline{u}(r)$. 
 As $u-\underline{u}(r)$ is always nonnegative in $Q_r$, we can omit the average term and this yields an equivalent form of \eqref{eq:supbddhobstacleh}
\begin{align}\label{ineq:oscu1}
&\overline{u}(r/4) - \underline{u}(r) = \esssup_{ Q_{r/4}} (u-\underline{u}(r)) \lesssim \bigg [\fint_{Q_{r/2}} (u-\underline{u}(r))^{h} \, dx \bigg ]^{1/h}  +r 
\end{align}
for any $h\in(0,\infty).$

For the second case, applying  Proposition~\ref{prop:sup-estimate} with $\ell=\overline{\psi}(r)$ and local H\"older continuity of $\psi$, we get 
 \begin{align*}
&\esssup_{Q_{\theta r/2}} (u-\underline{u}(r)) \leq \esssup_{Q_{ \theta r/2}} (u-\underline{\psi}(r))_+ \leq \esssup_{Q_{\theta r/2}} (u-\overline{\psi}(r))_+ +  [\psi]_{\beta}  r^\beta \\
&\lesssim (1-\theta)^{-4nq^2} \bigg[ \bigg ( \fint_{Q_{r/2}} (u-\overline{\psi}(r))_+^q \, dx \bigg )^{1/q} + |(u-\overline{\psi}(r))_{Q_{r/2}}|  \bigg] + r+ r^\beta \\
&\lesssim (1-\theta)^{-4nq^2}  \bigg[ \bigg ( \fint_{Q_{r/2}} (u-\underline{u}(r))_+^q \, dx \bigg )^{1/q} + |(u-\overline{\psi}(r))_{Q_{r/2}}|  \bigg] + r^\beta.
\end{align*}
Let us briefly focus on the average term. Again, using the H\"older continuity of the obstacle and the inequality $\underline{u}(r) \leq \overline{\psi}(r)$ combined with H\"older inequality to increase the exponent we get
\begin{align*}
|(u-\overline{\psi}(r))_{Q_{r/2}}| &= \left | \fint_{Q_{r/2}} u- \overline{\psi}(r) \, dx \right | = \left | \fint_{Q_{r/2}} u- \underline{\psi}(r) + \underline{\psi}(r) - \overline{\psi}(r) \, dx \right | \\
&\leq \left | \fint_{Q_{r/2}} u- \underline{\psi}(r) \, dx \right | + |\underline{\psi}(r) - \overline{\psi}(r)| \leq  \fint_{Q_{r/2}} u- \underline{\psi}(r) \, dx + [\psi]_{\beta} r^{\beta} \\
&\leq \fint_{Q_{r/2}} (u - \overline{\psi}(r))_+ \, dx + 2 [\psi]_{\beta} r^\beta \leq \fint_{Q_{r/2}} u- \underline{u}(r) \, dx + 2[\psi]_{\beta} r^{\beta}
 \\
 & \leq \left [\fint_{Q_{r/2}} (u- \underline{u}(r))^{q} \, dx \right ]^{1/q} + 2[\psi]_{\beta} r^{\beta}.
\end{align*}
Combining the two previous estimates we have
\begin{align*}
\esssup_{Q_{ \theta r/2}} (u-\underline{u}(r)) &\lesssim (1-\theta)^{-4nq^2} \left  [ \fint_{Q_{r/2}} (u-\underline{u}(r))^q \, dx \right ]^{1/q} + r^\beta.
\end{align*}
Therefore, performing the iteration argument in the same way as the proof of \cite[Corollary 5.8]{HarHM18}, we obtain
\begin{align}\label{ineq:oscu2}
\esssup_{Q_{r/4}} (u-\underline{u}(r)) \lesssim \left [\fint_{Q_{r/2}} (u-\underline{u}(r))^h \, dx \right]^{1/h} + r^\beta
\end{align}
for any $h \in (0, \infty)$.

Combining the inequalities \eqref{ineq:oscu1}-\eqref{ineq:oscu2} 
 and choosing $h=h_0$ where $h_0$ is given in Proposition \ref{prop:inf-estimate} 
we get
\begin{align}
\label{eq:osc-u}
\begin{split}
\overline{u}(r/4)-\underline{u}(r)
&\lesssim \left [\fint_{Q_{r/2}} (u-\underline{u}(r))^{h_0} \, dx \right ]^{1/h_0} + r^{\beta}.
\end{split}
\end{align}
 Since for a fixed $r$, the function $u-\underline{u}(r)$ is a nonnegative solution to the $\mathcal{K}^{\phi}_{\psi-\underline{u}(r)}$-obstacle problem, Proposition \ref{prop:inf-estimate} yields that 
\begin{align*}
\left [ \fint_{Q_{r/2}} (u-\underline{u}(r))^{h_0} \, dx \right ]^{1/h_0} 
\lesssim \underline{u}(r/4) - \underline{u}(r) +r.
\end{align*}
Combining this with \eqref{eq:osc-u} we arrive at
\begin{align*}
\overline{u}(r/4) - \underline{u}(r) \leq C (\underline{u}(r/4) - \underline{u}(r)) + c r^{\beta}
\end{align*}
for some constants $C, c >0$.
Again there are two cases. First if $(C+1)(\underline{u}(r/4)-\underline{u}(r)) \leq \overline{u}(r)-\underline{u}(r)$, then
\begin{align*}
\overline{u}(r/4) - \underline{u}(r/4) \leq \overline{u}(r/4) - \underline{u}(r) \leq \dfrac{C}{C+1}(\overline{u}(r)-\underline{u}(r)) + c r^{\beta}.
\end{align*}
On the other hand, if $(C+1)(\underline{u}(r/4)-\underline{u}(r)) > \overline{u}(r)-\underline{u}(r)$, then
\begin{align*}
\overline{u}(r/4) - \underline{u}(r/4) &= \overline{u}(r/4) - \underline{u}(r) - (\underline{u}(r/4) - \underline{u}(r)) \\
&<\overline{u}(r/4) - \underline{u}(r) - \dfrac{1}{C+1} (\overline{u}(r)-\underline{u}(r)) \\
&\leq \dfrac{C}{C+1} (\overline{u}(r)-\underline{u}(r)).
\end{align*}
From both cases we arrive to a conclusion that
\begin{align*}
\osc(u,r/4) := \overline{u}(r/4)-\underline{u}(r/4) &\leq \dfrac{C}{C+1} (\overline{u}(r) - \underline{u}(r)) + cr^{\beta} \\
&= \dfrac{C}{C+1} \osc(u,r) + cr^{\beta}.
\end{align*}
\end{proof}

\section{Higher regularity of the solution}
For the rest of the paper we assume that $\psi \in C^{1,\beta}_{\loc}(\Omega)$ for some $\beta \in (0,1)$. Then we note that  for any $ \Omega' \Subset \Omega,$ there exists a constant $ [\nabla\psi]_{\beta}>0$ such that 
$$ |\nabla \psi (x) - \nabla\psi(y)| \leq [\nabla\psi]_{\beta}|x-y|^{\beta} \ \ \text{for all $x, y \in \Omega'$.}$$

In the previous section, we proved the first result in Theorem~\ref{thm:C^alpha} with cubes because of the simple proof based on the previous results, especially, supremum and infimum estimates with cubes. However, from now on we use balls instead of cubes in order to prove our second result in Theorem~\ref{thm:C^1,alpha}. 
It can be also proved with cubes by the same approach we use here, but seems to be more complicated because in the process of its proof it has to mix cubes and balls when dealing with Calder\'on--Zygmund type estimates.

\subsection{Higher integrability of the gradient} We start this section by introducing some higher integrability results of $\nabla u$.

\begin{lem} [Higher Integrability, Theorem~1.1 \cite{HarHK17}]
Let $\phi \in \Phi_w(\Omega)$ satisfy  \azero{}, \aone{}, \ainc{p} and \adec{q} with constant $L \ge 1$ and $1<p\le q <\infty$. If $u \in W^{1,\phi}_{\loc}(\Omega)$ is a  local minimizer of  the $\phi$-energy \eqref{mainfnal}
then there exist $\sigma_0 = \sigma_0(n,p,q, L)>0$ and $c_1=c_1(n,p,q,L)\geq1$ such that 
\begin{equation}\label{ineq: higher integra_0}
\left( \fint_{B_r} \phi(x,|\nabla u|)^{1+\sigma_0} \,dx \right)^{\frac{1}{1+\sigma_0}} \leq c_1 \left(  \fint_{B_{2r}} \phi(x,|\nabla u|) \,dx+1 \right)
\end{equation}
for any $B_{2r} \Subset \Omega$ with $\| \nabla u \|_{L^{\phi}(B_{2r})} \leq 1.$
\end{lem}

\begin{lem}[Reverse H\"older type inequality, Lemma~4.7 \cite{HasO_pp18}]
Assume that $u \in W^{1,\phi}_{\loc}(\Omega)$ satisfies \eqref{ineq: higher integra_0} for some $B_{2r}\Subset \Omega$. For every $t \in (0,1]$ there exists $c = c(c_1,t,q)>0$ such that 
$$\left( \fint_{B_r} \phi(x,|\nabla u|)^{1+\sigma_0} \,dx \right)^{\frac{1}{1+\sigma_0}}  \leq c \bigg[ \bigg( \fint_{B_{2r}} \phi(x,|\nabla u|)^{t} \,dx\bigg)^{\frac1t}  +1 \bigg].$$
Moreover, if $\phi$ satisfies \azero{}, \aone{} and \adec{q} with $L \ge 1$ and $q>1$, and $\| \nabla u \|_{L^{\phi}(B_{2r})} \leq 1,$ then 
$$ \fint_{B_r} \phi(x,|\nabla u|) \,dx \leq \left( \fint_{B_r} \phi(x,|\nabla u|)^{1+\sigma_0} \,dx \right)^{\frac{1}{1+\sigma_0}}   \leq c  \left(\phi^{-}_{B_{2r}} \left( \fint_{B_{2r}} |\nabla u| \,dx\right)  +1 \right)$$
for some $c = c(c_1,q,L) \ge 1.$
\end{lem}

The next Lemma generalizes previous results to the obstacle case. The proof is a standard combination of Caccioppoli inequality, Sobolev--Poincar\'e inequality and Gehring's lemma.

\begin{lem}\label{lem:higherineq}
Let $\phi \in \Phi_w(\Omega) \cap C^1([0,\infty))$ satisfy  \azero{}, \aone{}, \ainc{p} and \adec{q} with constant $L \ge 1$ and $1<p\le q <\infty$. 
Assume that  $u$ is a solution to the $\mathcal{K}^{\phi}_{\psi}(\Omega)$-obstacle problem and $B_{2r}\Subset \Omega$ with $r>0$ satisfying  $\rho_{L^\phi(B_{2r})}(|\nabla u|),\rho_{L^\phi(B_{2r})}(|\nabla \psi|)  \le 1$.
Then  there exists $\sigma_0 \in (0,1)$ such that  
$\phi(\cdot, |\nabla u|) \in L^{1+\sigma_0}(B_r)$, and for any $\sigma \in (0, \sigma_0]$,
\begin{align}
\label{eq:higher-ineq-obs}
\fint_{B_r} \phi(x, |\nabla u|)^{1+\sigma}\,dx \leq c_2 \left [ \left( \fint_{B_{2r}}  \phi(x, |\nabla u|)\,dx \right)^{1+\sigma} +   \fint_{B_{2r}}  \phi(x, |\nabla \psi|)^{1+\sigma}\,dx +1 \right ]
\end{align}
for some $c_2=c_2(n,p,q,L)>0.$
\end{lem}
\begin{proof}
Let $\tau \in C^{\infty}_0(B_{2r})$ be a cut off function such that $ 0\leq \tau \leq 1, \tau \equiv 1 $ in $B_r$ and $|\nabla \tau| \leq \frac{c(n)}{r}.$
Since $\psi-\bar{\psi}_{B_{2r}} - u+\bar{u}_{B_{2r}} \geq \psi - u,$ we note that $\eta:= \tau^{q} ( \psi-\bar{\psi}_{B_{2r}} - u+\bar{u}_{B_{2r}}) \geq \psi-u$ in $B_{2r}$ where $q$ is given in \adec{q}.
Then we take $\eta$ as a test function  in \eqref{weakformobstacle} to have that
\begin{align*}
\int_{B_{2r}}  \partial \phi(x, \nabla u)  \cdot \nabla \eta \,dx \geq 0,
\end{align*}
which implies that
\begin{align*}
\int_{B_{2r}}  \left[ \partial \phi(x, \nabla u)  \cdot \nabla u \right] \tau^q  \,dx
&\leq \int_{B_{2r}}  \left[ \partial \phi(x, \nabla u)  \cdot \nabla \psi  \right] \tau^q  \,dx\\
& \quad+ q \int_{B_{2r}}  \left[\partial \phi(x, \nabla u)  \cdot \nabla \tau \right]\tau^{q-1} ( \psi-\bar{\psi}_{B_{2r}} - u+\bar{u}_{B_{2r}})  \,dx.
\end{align*}
Then applying Lemma \ref{est:Orlicz} (2) and  Lemma \ref{lem:phi-estimates} (4)--(5),
\begin{align*}
&\int_{B_{r}}  \phi(x,|\nabla u|)  \,dx\lesssim\int_{B_{2r}}  \left[ \partial \phi(x, \nabla u)  \cdot \nabla u \right] \tau^q  \,dx\\
& \leq \int_{B_{2r}}  \left[ \partial \phi(x, \nabla u)  \cdot \nabla \psi  \right] \tau^q  \,dx\\
& \quad+ q \int_{B_{2r}}  \left[\partial \phi(x, \nabla u) \cdot \nabla \tau \right]\tau^{q-1} ( \psi-\bar{\psi}_{B_{2r}} - u+\bar{u}_{B_{2r}})  \,dx\\
&\lesssim \kappa \int_{B_{2r}} \phi^*(x,  | \partial \phi(x, \nabla u)|    ) \,dx + c(\kappa) \int_{B_{2r}} \phi(x,  |\nabla \psi| \tau^q ) \,dx\\
&\quad + q \kappa \int_{B_{2r}} \phi^*(x,  | \partial \phi(x, \nabla u)| ) \,dx + q c(\kappa) \int_{B_{2r}} \phi(x, \tau^{q-1}  |\nabla \tau| |\psi-\bar{\psi}_{B_{2r}} - u+\bar{u}_{B_{2r}}|) \,dx\\
&\lesssim  \kappa \int_{B_{2r}} \phi(x, |\nabla u|) \,dx +\int_{B_{2r}} \phi(x,  |\nabla \psi|) \,dx\\
&\quad + \int_{B_{2r}} \phi\left(x, \frac{  |u-\bar{u}_{B_{2r}}|}{r}\right) \,dx+  \int_{B_{2r}} \phi\left(x, \frac{  |\psi-\bar{\psi}_{B_{2r}}|}{r}\right) \,dx
\end{align*}
for any $\kappa \in (0,1).$
Therefore,  by choosing $\kappa$ sufficiently small, we have 
\begin{align*}
\fint_{B_{r}}  \phi(x,|\nabla u|)  \,dx &\lesssim \fint_{B_{2r}} \phi\left(x, \frac{  |u-\bar{u}_{B_{2r}}|}{r}\right) \,dx \\
&\quad + \fint_{B_{2r}} \phi(x,  |\nabla \psi|) \,dx+  \fint_{B_{2r}} \phi\left(x, \frac{  |\psi-\bar{\psi}_{B_{2r}}|}{r}\right) \,dx.
\end{align*}

Here, Sobolev--Poincar\'e inequality \eqref{SPineq} yields that  for $1<s \leq p$,
\begin{align*}
\fint_{B_{2r}} \phi\left(x, \frac{  |u-\bar{u}_{B_{2r}}|}{r}\right) \,dx  \lesssim 
 \left( \fint_{B_{2r}} \phi\left(x, |\nabla u|\right)^{\frac1s} \,dx\right)^s + 1,
\end{align*}
because 
$$ \int_{B_{2r}} \phi\left(x, |\nabla u|\right)^{\frac1s} \,dx \leq \int_{B_{2r}} \phi\left(x, |\nabla u|\right) \,dx +1 \leq 2.$$
Similarly, we have that 
\begin{align*}
\fint_{B_{2r}} \phi\left(x, \frac{  |\psi-\bar{\psi}_{B_{2r}}|}{r}\right) \,dx  \lesssim 
 \left( \fint_{B_{2r}} \phi\left(x, |\nabla \psi|\right)^{\frac1s} \,dx\right)^s + 1 \leq  \fint_{B_{2r}} \phi\left(x, |\nabla \psi|\right) \,dx +1,
\end{align*}
because $$ \int_{B_{2r}} \phi\left(x, |\nabla \psi|\right)^{\frac1s} \,dx \leq \int_{B_{2r}} \phi\left(x, |\nabla \psi|\right) \,dx +1 \leq 2.$$

Hence, we conclude that 
 \begin{align*}
\fint_{B_{r}}  \phi(x,|\nabla u|)  \,dx &\lesssim    \left( \fint_{B_{2r}} \phi\left(x, |\nabla u|\right)^{\frac1s} \,dx\right)^s  + \fint_{B_{2r}} \phi(x,  |\nabla \psi|) \,dx +1.
\end{align*}
Since $\nabla \psi \in L^{\infty}_{\loc}(\Omega),$  by Gehring's lemma (see \cite[Theorem 6.6 and Corollary 6.1, pp. 203–204]{Giusti}), 
we obtain the desired estimates, when we denote the implicit constant by $c_2$.
\end{proof}

We can upgrade the higher integrability for the obstacle problem like in the case for minimizers. 

\begin{lem}[Reverse H\"older type inequality]\label{lem:reverse}

Under the same hypotheses as in Lemma~\ref{lem:higherineq}, for every $t \in (0,1]$, there exists $c_t = c_t(c_2, t,q) >0$ such that 

\begin{align}
\label{est:reverse-typse-inequality1}
\begin{split}
&\left (\fint_{B_r} \phi(x, |\nabla u|)^{1+\sigma_0} \, dx\right )^{1/(1+\sigma_0)}\\
&\quad \leq c_t \left [ \left (\fint_{B_{2r}} \phi(x, |\nabla u|)^t \, dx \right )^{1/t} + \left( \fint_{B_{2r}} \phi(x, |\nabla \psi|)^{1+\sigma_0} \, dx\right )^{1/(1+\sigma_0)} + 1 \right ].
\end{split}
\end{align}
Additionally, if $\phi$ satisfies  \azero{}, \aone{}, \adec{q} with constant $L \ge 1$ and $q>1$, and $\|\nabla u\|_{L^\phi(B_{2r})} \leq 1$, then

\begin{align}
\label{est:reverse-typse-inequality}
\begin{split}
\fint_{B_r} \phi(x, |\nabla u|)\, dx 
&\leq \left (\fint_{B_r} \phi(x, |\nabla u|)^{1+\sigma_0} \, dx \right )^{1/(1+\sigma_0)} \\
&\leq c\left [\phi^-_{B_{2r}} \left ( \fint_{B_{2r}} |\nabla u| \, dx \right ) +\left( \fint_{B_{2r}} \phi(x, |\nabla \psi|)^{1+\sigma_0} \, dx\right)^{1/(1+\sigma_0)}+ 1 \right ],
\end{split}
\end{align}
where $c\geq 1$ depends on $c_2, q$ and $L$.
\end{lem}
\begin{proof}
The first inequality \eqref{est:reverse-typse-inequality1} follows from \eqref{eq:higher-ineq-obs} by the same argument as in \cite[Remark~6.12, pp. 205]{Giusti}. 

The second inequality \eqref{est:reverse-typse-inequality} can be derived by the same way as in the proof of Lemma~4.7 of \cite{HasO_pp18}, We choose $t= 1/q$ in the inequality \eqref{est:reverse-typse-inequality1} to obtain
\begin{align*}
&\left (\fint_{B_r} \phi(x, |\nabla u|)^{1+\sigma_0} \, dx \right )^{1/(1+\sigma_0)}\\
& \lesssim \left (\fint_{B_{2r}} \phi^{+}_{B_{2r}} (|\nabla u|)^{1/q} \, dx \right )^{q} + \left( \fint_{B_{2r}} \phi(x, |\nabla \psi|)^{1+\sigma_0} \, dx\right )^{1/(1+\sigma_0)} + 1.
\end{align*}
As $\phi$ satisfies \adec{q}, we immediately see that $\phi^{+}(t)^{1/q}$ satisfies \adec{1}. Thus using Lemma \ref{lem:phi-estimates} (2) and Jensen's inequality for concave functions we arrive at
\begin{align}
\label{eq:used-concave}
\begin{split}
&\left (\fint_{B_r} \phi(x, |\nabla u|)^{1+\sigma_0} \, dx \right )^{1/(1+\sigma_0)} \\
&\lesssim \phi^{+}_{B_{2r}} \left (\fint_{B_{2r}} |\nabla u| \, dx \right ) + \left( \fint_{B_{2r}} \phi(x, |\nabla \psi|)^{1+\sigma_0} \, dx\right )^{1/(1+\sigma_0)}  + 1.
\end{split}
\end{align}
Now we want to use \aone{}, for which we need that 
\begin{align*}
\phi^{-}_{B_{2r}}\left (\fint_{B_{2r}} |\nabla u| \, dx \right ) \leq \dfrac{1}{|B_{2r}|},
\end{align*}
but this follows from Jensen's inequality since $\varrho_{\phi^-_{B_{2r}}}(|\nabla u|) \leq \varrho_{\phi}(|\nabla u|) \leq 1$ by assumption that $\|\nabla u\|_{L^\phi(B_{2r})} \leq 1$. Then if $\phi^{-}_{B_{2r}}\left (\fint_{B_{2r}} |\nabla u| \, dx \right )\geq 1$, \aone{}  yields that
\begin{align*}
\phi^{+}_{B_{2r}}\left (\fint_{B_{2r}} |\nabla u| \, dx \right ) \lesssim \phi^{-}_{B_{2r}}\left (\fint_{B_{2r}} |\nabla u| \, dx \right ).
\end{align*}
In the other case, that is $\phi^{-}_{B_{2r}}\left (\fint_{B_{2r}} |\nabla u| \, dx \right )< 1$ the constant term dominates in \eqref{est:reverse-typse-inequality}. 
Then $\phi^{+}_{B_{2r}}\left (\fint_{B_{2r}} |\nabla u| \, dx \right )$ can be estimated by the constant $c$ with the help of \adec{} and  \azero{}.
\end{proof}

\subsection{Comparison functions}
Now we are ready to prove some essential estimates for obtaining higher regularity. The first step is to construct a suitable reference problem introducing a regularized Orlicz function $\tilde \phi$. We use the function $\tilde \phi$ for Orlicz type equations, which are known to have solutions with $C^{1,\alpha}_\loc$-regularity.
%

For the rest of the paper, we assume that $\phi \in \Phi_c(\Omega) \cap C^1([0,\infty)) $ satisfies  \wvaone{}   with $\phi'$ satisfying  \inc{p-1} and \dec{q-1} for some $1<p\leq q$.

From  \wvaone{},
we note that $\phi$ satisfies a stronger version of \aone{} i.e., there exist $L \geq 1 $ and a non-decreasing, bounded, continuous function $\omega: [0, \infty) \rightarrow  [0,1]$ with $\omega(0)=0$ such that for any small ball $B_r \Subset \Omega$
\[
\phi^{+}_{B_r}(t) \le L \phi^{-}_{B_r} (t) 
\quad\text{for all}\quad t >0  \quad\text{with}\quad
\phi^{-}_{B_r}(t) \in \bigg[\omega(r), \frac{1}{|B_r|}\bigg].
\]
We further assume that $\phi'$ satisfies \azero{} with the same constant $L\ge 1$,  \inc{p-1} and \dec{q-1} for some $1<p\leq q$.

In order to use higher integrability results we need to assume that $|\nabla u|$ and $|\nabla \psi|$ are small in sense of norms. Recall that we assume $\psi$ to have continuous gradient, so the gradient has higher integrability over a compact set. These are achieved by considering as the integration domain a small enough ball. To quantify this smallness we henceforth fix  $\Omega' \Subset \Omega$ and consider $B_{2r}=B(x_0,2r) \subset \Omega'$  with $r>0$ so that 

\begin{align}
\label{ass:rsmall}
\begin{split}
&r \leq \dfrac{1}{2},\ \  \omega(2r) \leq \frac1L \ \  \textrm{ and }\\
& |B_{2r}| \leq \min \bigg  \lbrace \frac{1}{2L}, 2^{-\frac{2(1+\sigma_0)}{\sigma_0}} \left (\int_{\Omega'} \phi(x, |\nabla u|)^{1+\sigma_0} \; dx \right )^{-\frac{2+\sigma_0}{\sigma_0}},2^{-\frac{2(1+\sigma_0)}{\sigma_0}} \left (\int_{\Omega'} \phi(x, |\nabla \psi|)^{1+\sigma_0} \; dx \right )^{-\frac{2+\sigma_0}{\sigma_0}} \bigg \rbrace,
\end{split}
\end{align}
where $\sigma_0 \in (0,1)$ is given in Lemma~\ref{lem:higherineq}. 
Hence we have from the above assumption \eqref{ass:rsmall} that 
\begin{align}
\label{eq:highernormsmall}
\begin{split}
&\int_{B_{2r}} \phi(x, |\nabla u|) \, dx \leq \int_{B_{2r}} \phi(x, |\nabla u|)^{1+\frac{\sigma_0}2} +1 \, dx \\
& \leq |B_{2r}|\left (\fint_{B_{2r}} \phi(x,|\nabla u|)^{1+\sigma_0} \, dx \right )^{\frac{2+\sigma_0}{2(1+\sigma_0)}} + |B_{2r}| \leq \dfrac{1}{2} + \dfrac{1}{2} =1,
\end{split}
\end{align}
and similarly,
\begin{align}
\label{eq:highernormsmall}
\begin{split}
&\int_{B_{2r}} \phi(x, |\nabla \psi|) \, dx \leq \int_{B_{2r}} \phi(x, |\nabla \psi|)^{1+\frac{\sigma_0}2} +1 \, dx \\
& \leq |B_{2r}|\left (\fint_{B_{2r}} \phi(x,|\nabla \psi|)^{1+\sigma_0} \, dx \right )^{\frac{2+\sigma_0}{2(1+\sigma_0)}} + |B_{2r}|   \leq \dfrac{1}{2} + \dfrac{1}{2} =1
\end{split}
\end{align}
by applying  H\"older's inequality. 
Then the hypotheses in Lemma~\ref{lem:higherineq} are satisfied and so it follows that for a solution $u$ to the $\mathcal{K}_{\psi}^{\phi}(\Omega)$-obstacle problem, $\phi(\cdot, |\nabla u|) \in L_{loc}^{1+\sigma_0}(\Omega)$ and \eqref{eq:higher-ineq-obs} holds.


For comparison argument, we shall use the regularized Orlicz  function $\tilde \phi$ constructed from a generalized Orlicz function $\phi$, which was introduced in \cite{HasO_pp18}.
We now give its definition and list its properties that will be used in our proof; see \cite[Section 5]{HasO_pp18} for 
its detailed construction and the proofs of the properties. 

We start by defining 
\begin{align}
\label{t1t2}
\phi^{\pm}(t):= \phi^{\pm}_{B_{2r}}(t), \quad
t_1 := (\phi^-)^{-1}(\omega(2r)) \quad \text{and} \quad t_2:=(\phi^-)^{-1}(|B_{2r}|^{-1}).
\end{align}
Note that $t_1 \leq 1\leq t_2$ from the assumptions \eqref{ass:rsmall} and \azero{}.
Now we can define
\begin{align}
\label{defn:phi_B}
\tau_{B_{2r}}(t):=
\begin{cases}
a_1\left (\frac{t}{t_1}\right )^{p-1}, &\text{if } 0 \leq t < t_1, \\
\phi'(x_0,t) &\text{if } t_1\leq t \leq t_2, \\
a_2\left (\frac{t}{t_2}\right )^{p-1}, &\text{if } t_2 < t < \infty,
\end{cases}
\end{align}
where $a_1 := \phi'(x_0,t_1)$ and $a_2 := \phi'(x_0,t_2)$ are chosen so that $\tau_{B_{2r}}$ is continuous. We continue defining
\begin{align*}
\phi_{B_{2r}}(t):= \int_{0}^t \tau_{B_{2r}}(s) \, ds.
\end{align*}
Finally, for $\eta \in C^{\infty}_0(\R)$ with $\eta \geq 0$, $\supp \eta \subset (0,1)$ and $\|\eta\|_1=1$, we define the regularized Orlicz function $\tilde \phi(0):=0$ and
\begin{align}
\label{defn:tilde-phi}
\tilde \phi(t) := \int_{0}^{\infty} \phi_{B_{2r}}(t\sigma) \eta_{r}(\sigma-1) \, d\sigma = \int_{0}^\infty \phi_{B_{2r}}(s) \eta_{rt}(s-t) \, ds, \text{ where } \eta_r(t):= \tfrac 1r \eta\left ( \tfrac tr \right ).
\end{align}

This function $\tilde \phi$ has the following properties. 
\begin{lem}[Lemma 5.10 \cite{HasO_pp18}]
\label{lem:tilde-phi}

 Let $\tilde \phi$ be the regularized Orlicz function. Then
\begin{enumerate}
\item{$\phi_{B_{2r}}(t) \leq \tilde{\phi}(t) \leq (1+cr) \phi_{B_{2r}}(t)$ for all $t>0$ with $c>0$ depending only on $q$, and $0 \leq \tilde \phi(t)-\phi(x_0,t) \leq c r \phi^-(t)+c \omega(2r) \leq c \phi^-(t)$ for all $t \in [t_1,t_2]$; }
\medskip
\item{$\tilde \phi \in C^1([0,\infty))$ and it satisfies \inc{p}, \dec{q} and \azero{}, and $\tilde \phi'$ satisfies \inc{p-1} and \dec{q-1} and  \azero{}. In particular $\tilde \phi'(t) \approx t \tilde \phi''(t)$ for all $t>0$;}
\medskip
\item{$\tilde \phi(t) \leq c \phi(x,t)$ for all $(x,t) \in B_{2r} \times [1,\infty)$, and so $\tilde \phi(t) \lesssim \phi(x,t)+1$ for all $(x,t) \in B_{2r} \times [0,\infty).$}
\end{enumerate}
Here the constant $c$ and the implicit constants depend only on $n,p,q$ and $L$.
\end{lem}


Now let us consider the following comparison principle for $ \tilde \phi$.
 \begin{lem}\label{lem:psi<w}
Assume that $w \in W^{1,\phi}(\Omega)$ satisfies 
 \begin{equation*}
  \left\{\begin{array}{rclcc} 
  -\mathrm{div} \left( \frac{\tilde{\phi}' (|\nabla \psi|)}{|\nabla \psi |} \nabla \psi \right) &  \leq  &  -\mathrm{div}  \left( \frac{\tilde{\phi}' (|\nabla w|)}{|\nabla w |} \nabla w \right)  & \text{ in } & B_{r},\\ \psi& \leq & w & \text{ on } & \partial B_{r}, 
  \end{array}\right.
   \end{equation*}
   in the weak sense, that is, $(\psi - w)_+ \in W^{1,\tilde{\phi}}(B_r)$ and  
   $$ \int_{B_r} \left(  \frac{\tilde{\phi}' (|\nabla \psi|)}{|\nabla \psi |} \nabla \psi  -  \frac{\tilde{\phi}' (|\nabla w|)}{|\nabla w |} \nabla w \right) \cdot \nabla \eta \, dx \leq 0\ \  \text{ for all $\eta \in W^{1,\tilde{\phi}} (B_r)$ with } \eta \geq 0.$$
   Then we have $\psi \leq w $ a.e. in $B_r.$
\end{lem}
\begin{proof}
By taking $\eta = (\psi - w)_+$ as a test function to the above weak formulation,  
we see 
$$  \int_{B_r\cap \{\psi > w \}} \left(  \frac{\tilde{\phi}' (|\nabla \psi|)}{|\nabla \psi |} \nabla \psi  -  \frac{\tilde{\phi}' (|\nabla w|)}{|\nabla w |} \nabla w \right) \cdot  (\nabla \psi - \nabla w) \, dx \leq 0.$$
Then using first (4) and then (2) of Lemma \ref{est:Orlicz}, we obtain
\begin{align*}
& \int_{B_r \cap \{\psi > w \}} \tilde{\phi}( |\nabla \psi - \nabla w |)\,dx \\
 &  \lesssim  \kappa   \int_{B_r \cap \{\psi > w \}} \tilde{\phi}( |\nabla \psi  |) + \tilde{\phi}( | \nabla w |)  \,dx \\
 &\quad + \kappa^{-1} \int_{B_r \cap \{\psi > w \}} \frac{\tilde{\phi}'( |\nabla \psi |+| \nabla w |)}{|\nabla \psi |+| \nabla w |} |\nabla \psi - \nabla w|^2 \,dx   \\
 &\lesssim  \kappa   \int_{B_r \cap \{\psi > w \}} \tilde{\phi}( |\nabla \psi  |) + \tilde{\phi}( | \nabla w |)  \,dx \\
 &\quad + \kappa^{-1} \int_{B_r \cap \{\psi > w \}} \left ( \dfrac{\tilde \phi'(|\nabla \psi|)}{|\nabla \psi|} \nabla \psi - \dfrac{\tilde \phi'(\nabla w|)}{|\nabla w|}\nabla w \right ) \cdot (\nabla \psi - \nabla w) \,dx   \\
 &  \leq \kappa   \int_{B_r \cap \{\psi > w \}} \tilde{\phi}( |\nabla \psi  |) + \tilde{\phi}( | \nabla w |)  \,dx
 \end{align*}
 for any $\kappa \in (0,1).$ Since $\kappa$ is arbitrary, we have that $\psi \leq w $ a.e. in $B_r.$
\end{proof}

Next we define two equations and corresponding solutions $w$ and $v$ to which we compare our 
solution $u$ to the obstacle problem \eqref{weakformobstacle}. We will prove some energy estimates of $w$ and $v$ with respect to $u$ and regularized Orlicz function $\tilde \phi$ in Lemma \ref{lem:wv<u}.

Let $u \in \mathcal{K}^{\phi}_{\psi}(\Omega)$ be a solution to the $\mathcal{K}^{\phi}_{\psi}(\Omega)$-obstacle problem and $B_{2r}\subset \Omega$ with $r>0$ satisfying \eqref{ass:rsmall}. 
Let us consider the unique weak solution $w \in W^{1,\tilde{\phi}}(B_r)$  of 
 \begin{equation}\label{eq:wpsi}
  \left\{\begin{array}{rclcc} 
  -\mathrm{div} \left( \frac{\tilde{\phi}' (|\nabla w|)}{|\nabla w |} \nabla w \right) &  =  &  -\mathrm{div}\left( \frac{\tilde{\phi}' (|\nabla \psi|)}{|\nabla \psi |} \nabla \psi \right)  & \text{ in } & B_{r},\\ w& = & u & \text{ on } & \partial B_{r},
  \end{array}\right.
   \end{equation}
   and 
  the unique weak solution  $v \in W^{1,\tilde{\phi}}(B_r)$  of 
 \begin{equation}\label{eq:vw}
  \left\{\begin{array}{rclcc} 
  -\mathrm{div}\left( \frac{\tilde{\phi}' (|\nabla v|)}{|\nabla v |} \nabla v \right) &  =  &  0 & \text{ in } & B_{r},\\ v& = & w& \text{ on } & \partial B_{r}.
  \end{array}\right.
   \end{equation} 
From now on, we set $\overline{\nabla \psi}: = \sup_{\Omega}|\nabla \psi|$ for simplicity.

\begin{lem}\label{lem:wv<u}
For $w$ and $v$ defined in \eqref{eq:wpsi} and \eqref{eq:vw} we have the energy estimates:
\begin{align}
\label{est:energy}
 \fint_{B_r} \tilde{\phi}(|\nabla w|) \,dx \le c \left(  \fint_{B_r} \tilde{\phi}(|\nabla u|) \,dx +  \fint_{B_r} \tilde{\phi}(|\nabla \psi|) \,dx \right) \leq c \left( \fint_{B_r} \tilde{\phi}(|\nabla u|) \,dx  +1\right)
\end{align}
and
\begin{align}
\label{est:energy2}
  \fint_{B_r} \tilde{\phi}(|\nabla v|) \,dx \le    c  \fint_{B_r} \tilde{\phi}(|\nabla w|) \,dx \leq  c \left( \fint_{B_r} \tilde{\phi}(|\nabla u|) \,dx  +1\right),
\end{align}
where $c = c(n, p,q,L,\overline{\nabla \psi})>0.$
\end{lem}
\begin{proof}
Testing with $w-u \in W^{1,\tilde \phi}_0(B_r)$  to the weak
formulation, the equation \eqref{eq:wpsi} yields
\begin{align*}
\fint_{B_r} \dfrac{\tilde \phi'(|\nabla w|)}{|\nabla w|} \nabla w \cdot \nabla w\, dx = \fint_{B_r} \dfrac{\tilde \phi'(|\nabla w|)}{|\nabla w|} \nabla w \cdot \nabla u \, + \dfrac{\tilde \phi'(|\nabla \psi|)}{|\nabla \psi|}\nabla \psi \cdot \nabla (w-u) \, dx.
\end{align*}
Applying $\tilde \phi'(t) t \approx \tilde \phi(t)$ in Lemma~\ref{est:Orlicz} (1), Lemma \ref{lem:phi-estimates} (4)-(5), 
and \adec{q}, we obtain
\begin{align*}
\fint_{B_r} \dfrac{\tilde \phi'(|\nabla w|)}{|\nabla w|} \nabla w \cdot \nabla w\, dx &\leq c \fint_{B_r} \dfrac{\tilde \phi(|\nabla w|)}{|\nabla w|} |\nabla u| + \dfrac{\tilde \phi(|\nabla \psi|)}{|\nabla \psi|} (|\nabla w|+ |\nabla u|) \, dx \\
&\leq \dfrac{1}{2} \fint_{B_r} \tilde \phi(|\nabla w|) \, dx + c \fint_{B_r} \tilde \phi(|\nabla u|) + \tilde \phi(|\nabla \psi|) \, dx.
\end{align*}
Recalling Lemma~\ref{est:Orlicz} (1) again, we can move the first term on the right-hand side to the left-hand side. Finally, estimating $|\nabla \psi|$ by $\overline{\nabla \psi}$ we have proven \eqref{est:energy}. 

The second energy estimate \eqref{est:energy2} follows immediately since $v$ as a solution minimizes the corresponding energy integral and $v$ and $w$ have the same boundary values in the Sobolev sense.
\end{proof}

As $v$ is the solution of the $\phi$-Laplacian equation, it is known to have $C^{1,\alpha_0}$-regularity for some $\alpha_0 >0$ from \cite{Lie91} (see also \cite[Lemma 4.12]{HasO_pp18}). Additionally, for any $B(x_0,\rho) \subset B_r$, we have
\begin{align}
\label{v-Lip}
\sup_{B(x_0,\rho/2)} |\nabla v| \leq c \fint_{B(x_0,\rho)} |\nabla v| \, dx
\end{align}
and for any $\tau \in (0,1)$
\begin{align}
\label{nablavtau}
\fint_{B(x_0,\tau \rho)} |\nabla v - (\nabla v)_{B(x_0, \tau \rho)}| \, dx \leq c \tau^{\alpha_0} \fint_{B(x_0,\rho)} |\nabla v| \, dx,
\end{align}
where $\alpha_0 \in (0,1)$ and $c=c(n,p,q)>0$.

\subsection{Calder\'on--Zygmund type estimates}
We also need the following reverse type estimate and Calder\'on-Zygmund type estimate for the equation \eqref{eq:wpsi}. Here we rely on similar results in \cite[Lemma 4.15]{HasO_pp18}. As the proofs of these results are almost identical to the original one, we give just sketches of the proofs. The second lemma is an application suitable for the obstacle problem of the more abstract Calder\'on--Zygmund type estimates outlined in the first lemma.

The following is Calder\'on-Zygmund type estimates with  a ball $B_r$ with radius $r$.

\begin{lem}
Let $\phi \in \Phi_c \cap C^1([0,\infty)) \cap C^2((0,\infty))$ with $\phi'$ satisfying \inc{p-1} and \dec{q-1} for some $1<p\leq q$, and $|B_r|\leq 1$. If $w$ is a solution to \eqref{eq:wpsi}, then there exists a constant $c=c(n,p,q,p_1,q_1,L)>0$ such that
\begin{align}
\label{eq:CZ-norm}
\|\phi(|\nabla w|)\|_{L^{\theta}(B_r)} \leq c \left ( \|\phi(|\nabla \psi|)\|_{L^{\theta}(B_r)} + \|\phi(|\nabla u|)\|_{L^{\theta}(B_r)}\right )
\end{align}
for any $\theta \in \Phi_w(B_r)$ satisfying \azero, \aone, \ainc{p_1} and \adec{q_1} with constant $L\geq 1$ and $1<p_1\leq q_1$.

Moreover, fix $\kappa >0$ and assume that $\int_{B_r} \theta(x, \phi(|\nabla u|))+\theta (x,\phi(|\nabla \psi|)) \, dx \leq \kappa$. Then
\begin{align}
\label{eq:CZ-modular}
\fint_{B_r} \theta(x. \phi(|\nabla w|)) \,dx \leq c\left (\kappa^{\frac{q_1}{p_1}-1} + 1 \right ) \left (\fint_{B_r}\theta(x, \phi(|\nabla u|)) + \theta (x,\phi(|\nabla \psi|)) \, dx +1 \right )
\end{align}
for some $c\geq 0$ depending on $n, p, q, p_1, q_1$ and $L$.
\end{lem}

\begin{proof}
In the same way to \cite[Theorem B.1 in Appendix B]{HasO_pp18} (also see \cite{MenP12}), we obtain that for any $s \in (0,\infty)$ and Muckenhoupt weight $\mu\in A_s$, 
\begin{align}
\label{eq:CZ-weight}
\int_{\Omega} \phi(|\nabla w|)^s \mu(x)\,dx \leq c \left (\int_{\Omega} \phi(|\nabla \psi|)^s \mu(x)\,dx + \int_{\Omega}\phi(|\nabla u|)^s \mu(x)\,dx\right )
\end{align}
for some constant $c = c(n,p,q,s,[\mu]_{A_s})>0$
but we need to replace (B.2) by
\begin{align*}
\div \left (\frac{\phi'(|\nabla w|)}{|\nabla w|} \nabla w\right ) = \div \left (\frac{\phi'(|\nabla \psi|)}{|\nabla \psi|} \nabla \psi\right )\text{ in } \Omega \text{ with } w=u \text{ on } \partial \Omega,
\end{align*} 
where $\Omega\subset \mathbb{R}^n$ is a Reifenberg flat domain.
In addition, to compare this equation with an equation having zero boundary values on $\partial \Omega$ in a local region near boundary, 
we add the one more assumption  
$$ \fint_{\Omega_5} \phi(|\nabla \psi|)\,dx \le \delta$$
to (B.7) and  continue similarly with the proofs
 in Lemmas~B.5 and B.11 to derive \eqref{eq:CZ-weight}. 
Hence \eqref{eq:CZ-weight} with $\Omega=B_r$ implies the desired norm inequality \eqref{eq:CZ-norm} via the extrapolation result for the generalized Orlicz function in \cite[Corollary 5.3.4]{HarH19}.

The modular inequality \eqref{eq:CZ-modular} follows by modifying the proof of \cite[Lemma 4.15]{HasO_pp18} slightly, that is replacing the definition of constant $M$ with
\begin{align*}
M:= \left (\theta^{-}\right )^{-1} \left (\int_{B_r} \theta(x, \phi(|\nabla u|)) + \theta(x, \phi(|\nabla \psi|)) \, dx \right ).
\end{align*}
Then in similar fashion we obtain
\begin{align*}
\int_{B_r} \overline{\theta}(x,\overline{\phi}(|\nabla u)) + \overline{\theta}(x,\overline{\phi}(|\nabla \psi)) \, dx \leq 1 \Rightarrow \|\overline{\phi}(|\nabla u|)\|_{L^{\overline{\theta}}(B_r)} + \|\overline{\phi}(|\nabla \psi|)\|_{L^{\overline{\theta}}(B_r)} \leq 2
\end{align*}
and the rest follows line by line.
\end{proof}

\begin{lem}

We have that  
\begin{equation}\label{est:reverse} 
\left(\fint_{B_{r}} \phi(x, |\nabla u|)^{1+\sigma_0} \,dx\right)^{\frac{1}{1+\sigma_0}}  \lesssim  \tilde{\phi}\left( \fint_{B_{2r}} |\nabla u|\,dx\right) +1
    \end{equation}
and
\begin{align}\label{est:CZwu}
\begin{split}
\fint_{B_r} \phi(x, |\nabla w|) \,dx  &\leq \left( \fint_{B_r} \phi(x, |\nabla w|)^{1+\frac{\sigma_0}2} \,dx\right)^{\frac{2}{2+\sigma_0}} \\
&\lesssim \left( \fint_{B_{r}} \phi(x, |\nabla u|)^{1+\frac{\sigma_0}2} \,dx
+1\right)^{\frac{2}{2+\sigma_0}},
\end{split} 
    \end{align}
    where the implicit constant depends on $n, p, q, L$ and $\overline{\nabla \psi}$. 

\end{lem}
\begin{proof}
Since \wvaone{} implies \aone,  we have from \eqref{est:reverse-typse-inequality} in Lemma \ref{lem:reverse} 
that
 \begin{align*}
&\left (\fint_{B_{r}} \phi(x, |\nabla u|)^{1+\sigma_0} \, dx \right )^{1/(1+\sigma_0)}\\
& \quad\qquad\lesssim \phi^{-}_{B_{2r}}\left (\fint_{B_{2r}} |\nabla u| \,dx\right ) +  \left( \fint_{B_{2r}} \phi(x, |\nabla \psi|)^{1+\sigma_0} \, dx  \right)^{1/(1+\sigma_0)} +1.
\end{align*}
Using $\overline{\nabla \psi}$ we can absorb the second term on the right-hand side to the implicit constant. 
When $\fint_{B_{2r}} |\nabla u| \, dx \leq 1$, we obtain 
 \eqref{est:reverse} by finiteness of $\phi^{-}_{B_{2r}}$. 
  In the other case that $\fint_{B_{2r}} |\nabla u| \, dx >1$,  we have 
\begin{align*}
1 < \fint_{B_{2r}} |\nabla u|  dx \leq (\phi^{-}_{B_{2r}})^{-1}\left(\fint_{B_{2r}} \phi^{-}_{B_{2r}}(|\nabla u|) dx \right) \leq (\phi^{-}_{B_{2r}})^{-1}(|B_{2r}|^{-1}) =t_2,
\end{align*}
where $t_2$ is defined  in \eqref{t1t2}. 
Thus, Lemma \ref{lem:tilde-phi} (1) yields that
\begin{align*}
&\left (\fint_{B_{r}} \phi(x,|\nabla u|)^{1+\sigma_0}\, dx \right )^{1/(1+\sigma_0)} \lesssim \phi^-_{B_{2r}} \left (\fint_{B_{2r}} |\nabla u| \, dx \right ) +1 \\
& \quad\qquad  \lesssim \phi\left (x_0, \fint_{B_{2r}} |\nabla u| \, dx\right ) +1 \lesssim \tilde \phi \left (\fint_{B_{2r}} |\nabla u| \, dx \right )+1
\end{align*}
 which is \eqref{est:reverse}.

In order to prove \eqref{est:CZwu}, let us consider the function $\theta \in \Phi_{w}(B_r)$ which is given in  
\begin{align*}
\theta(x,t) := [\phi(x, \tilde \phi^{-1}(t))]^{1+\frac{\sigma_0}{2}}
\end{align*}
for any fixed $\sigma_0 \in (0,1)$. Applying the Calder\'on-Zygmund estimates  \eqref{eq:CZ-modular} and local boundedness of $|\nabla \psi|$ we have that
\begin{align*}
\fint_{B_r} \phi(x, |\nabla w|)^{1+\frac{\sigma_0}2} \,dx & = \fint_{B_r} \theta(x, \tilde{\phi}(|\nabla w|))\,dx  \\
& \lesssim  \fint_{B_{r}} \theta(x, \tilde{\phi}(|\nabla u|))\,dx +\fint_{B_{r}} \theta(x, \tilde{\phi}(|\nabla \psi|))\,dx +1\\
& = \fint_{B_{r}} \phi(x, |\nabla u|)^{1+\frac{\sigma_0}2} \,dx +\fint_{B_{r}} \phi(x, |\nabla \psi|)^{1+\frac{\sigma_0}2} \,dx +1\\
&\lesssim \fint_{B_{r}} \phi(x, |\nabla u|)^{1+\frac{\sigma_0}2} \,dx +1.
   \end{align*}
 The desired inequality \eqref{est:CZwu} follows by taking the $\frac{2}{2+\sigma_0}$'th root from both sides and applying H\"older's inequality on the left hand side.
\end{proof}

\subsection{Comparison estimates} Now we are ready to prove comparison estimates between different solutions. These are the main ingredients for the proof of Theorem~\ref{thm:C^1,alpha}. Much of the proof is similar to the case without any obstacle and these omitted details can be found in \cite{HasO_pp18}.

\begin{lem}\label{lem:comparison-uw}
Suppose that $w$ is a solution to the equation \eqref{eq:wpsi}, $u$ is a solution to the equation \eqref{weakformobstacle} and $B_{2r} \Subset \Omega$. Then
$$ \fint_{B_r} |\nabla u - \nabla w| \,dx \leq c  (\omega(2r)^{p/q}+r^{\gamma} +r^{\beta})^{\frac1{2q}} \left(\fint_{B_{2r}} |\nabla u |\,dx+ 1\right)$$ 
for some $\gamma = (n,p,q,L) \in (0,1)$ and some constant $c = c(n,p,q,L,[\nabla \psi]_{\beta}, \overline{\nabla \psi})>0,$
where $\omega$ is from the assumption \wvaone{}.

\end{lem}
\begin{proof}
From \eqref{eq:wpsi}, we see that
 \begin{equation}\label{est:solcomobstacle}
  \int_{B_r}  \frac{\tilde{\phi}' (|\nabla w|)}{|\nabla w |} \nabla w \cdot (\nabla w-\nabla u) \,dx =   \int_{B_r}  \frac{\tilde{\phi}' (|\nabla \psi|)}{|\nabla \psi |} \nabla \psi \cdot (\nabla w-\nabla u) \,dx,
  \end{equation}
  by taking $w-u \in W^{1,\tilde \phi}_0(B_r)$ as a test function to the weak
formulation.

Recalling Lemma \ref{lem:psi<w} and setting $w = u $ in $\Omega \backslash B_r$,  we have that $w \in W^{1,\phi}(\Omega)$ and $w \geq \psi$ a.e. in $\Omega$, and so $w \in \mathcal{K}^{\phi}_{\psi}(\Omega).$
Therefore by taking $\eta:= w-u $
 as a test function in  \eqref{weakformobstacle}, 
 \begin{equation}\label{est:solobstacle}
 \int_{B_r}  \frac{\partial_t \phi  (x,|\nabla u|)}{|\nabla u |} \nabla u \cdot (\nabla w-\nabla u) \,dx  \geq 0.
\end{equation}

From Lemma \ref{est:Orlicz} (3) and \eqref{est:solcomobstacle}, we then derive that 
\begin{align*}
 &\int_{B_r } \tilde{\phi}''( |\nabla u |+| \nabla w |) |\nabla u - \nabla w|^2 \,dx\\
 & \lesssim \int_{B_r }\tilde{\phi}( |\nabla u  |) - \tilde{\phi}( | \nabla w |) -  \frac{\tilde{\phi}' (|\nabla w|)}{|\nabla w |} \nabla w  \cdot (\nabla u-\nabla w) \,dx \\
 & = \int_{B_r }\tilde{\phi}( |\nabla u  |) - \tilde{\phi}( | \nabla w |)\,dx +  \int_{B_r }\frac{\tilde{\phi}' (|\nabla \psi|)}{|\nabla \psi |} \nabla \psi  \cdot (\nabla w-\nabla u) \,dx \\
  & = \int_{B_r }\tilde{\phi}( |\nabla u  |) -\phi(x, |\nabla u  |)  \,dx+\int_{B_r } \phi(x, |\nabla u  |) -\phi(x, |\nabla w  |) \,dx  \\
  &\quad + \int_{B_r }\phi(x, |\nabla w  |) - \tilde{\phi}( | \nabla w |)\,dx  + \int_{B_r }\frac{\tilde{\phi}' (|\nabla \psi|)}{|\nabla \psi |} \nabla \psi  \cdot (\nabla w-\nabla u) \,dx.
    \end{align*}
Here, applying Lemma \ref{est:Orlicz} (3) pointwise again with \eqref{est:solobstacle}, we see that
\begin{align*}
&\int_{B_r}\phi(x, |\nabla u  |) -\phi(x, |\nabla w  |) \,dx \\
&\lesssim\int_{B_r} \frac{\partial_t \phi(x,|\nabla u|)}{|\nabla u |} \nabla u\cdot (\nabla u-\nabla w)\,dx - \int_{B_r} \frac{\partial_t \phi(x,|\nabla u| + |\nabla w|)}{|\nabla u| + |\nabla w|} |\nabla u - \nabla w|^2 \,dx \\
&\leq-\int_{B_r} \frac{\partial_t \phi(x,|\nabla u| + |\nabla w|)}{|\nabla u| + |\nabla w|} |\nabla u - \nabla w|^2 \,dx  \leq 0
    \end{align*}
    since $\phi$ is increasing by assumption.
    Combining the above two inequalities, we obtain that 
    \begin{align*}
     &\fint_{B_r } \tilde{\phi}''( |\nabla u |+| \nabla w |) |\nabla u - \nabla w|^2 \,dx\\
  & \leq \underbrace{\fint_{B_r }\tilde{\phi}( |\nabla u  |) -\phi(x, |\nabla u  |)  \,dx}_{=:I_1} +  \underbrace{\fint_{B_r }\phi(x, |\nabla w  |) - \tilde{\phi}( | \nabla w |)\,dx}_{=:I_2}  \\
   &\quad +\underbrace{ \fint_{B_r }\frac{\tilde{\phi}' (|\nabla \psi|)}{|\nabla \psi |} \nabla \psi  \cdot (\nabla w-\nabla u) \,dx.}_{:=I_3}
    \end{align*}

On the other hand, using Lemma \ref{est:Orlicz} (4) and \eqref{est:energy}, we infer that
\begin{align*}
&\fint_{B_r} \tilde{\phi}( | \nabla u - \nabla w|) \,dx  \\
&\lesssim \kappa   \fint_{B_r } \tilde{\phi}( |\nabla u  |) + \tilde{\phi}( | \nabla w |)  \,dx + \kappa^{-1} \fint_{B_r } \tilde{\phi}''( |\nabla u |+| \nabla w |) |\nabla u - \nabla w|^2 \,dx   \\
 & \lesssim \kappa \left(  \fint_{B_r } \tilde{\phi}( |\nabla u  |)  \,dx  + 1\right)+ \kappa^{-1}  (I_1+I_2+I_3)
\end{align*}
for any $\kappa \in (0,\infty).$

 To estimate $I_2$, we adopt the method in the proof of  \cite[Lemma 6.2]{HasO_pp18}.  We start by splitting $B_r$ into three parts:
\begin{align*}
E_1&:= B_r \cap \{\phi^-(|\nabla w|) \leq \omega(2r)\} \\
E_2 &:= B_r \cap \{ \omega(2r) < \phi^-(|\nabla w|) \leq |B_{2r}|^{-1+\varepsilon_0}\}\\
E_3&:= B_r \cap \{|B_{2r}|^{-1+\varepsilon_0} < \phi^{-}(|\nabla w|) \}.
\end{align*}
In the set $E_1$ assumptions \dec{q} and \azero{} of $\phi$ yield that $|\nabla w| \lesssim \omega(2r)^{\frac1q}$. Now \inc{p} and \azero{} of $\phi$ and $\tilde{\phi}$ imply
\begin{align}
\label{est:E_1}
\fint_{B_r} |\phi(x, |\nabla w|) - \tilde \phi(|\nabla w|)| \chi_{E_1}  \,dx \lesssim \omega(2r)^{p/q} \fint_{B_r}\chi_{E_1} \, dx \leq \omega(2r)^{p/q} .
\end{align}

In the set $E_2$ we get
\begin{align*}
\omega(2r) < \phi^-(|\nabla w|) \leq |B_{2r}|^{-1+\varepsilon_0} < |B_{2r}|^{-1}
\end{align*}
which implies that $t_1 < |\nabla w| < t_2$. Therefore we can use \wvaone{} and Lemma \ref{lem:tilde-phi} (1) to get
\begin{align*}
|\phi(x, |\nabla w|)-\tilde \phi(|\nabla w|)| &\leq |\phi(x, |\nabla w|)- \phi(x_0, |\nabla w|) | + |\phi(x_0, |\nabla w|) - \tilde \phi(|\nabla w|)|\\
&\leq \phi^+(|\nabla w|) - \phi^-(|\nabla w|) + \tilde \phi(|\nabla w|) - \phi(x_0, |\nabla w|) \\
&\lesssim (1+\omega(2r))\phi^-(|\nabla w|) +\omega(2r) - \phi^-(|\nabla w|) + r\phi^{-}(|\nabla w|) + \omega(2r) \\
& \lesssim (r+\omega(2r)) \phi^-(|\nabla w|) + \omega(2r).
\end{align*}
 Note that this is the only place, where the assumption \wvaone{} is needed. 
Now using \eqref{est:CZwu} and \eqref{est:reverse} we see that
\begin{align}
\label{est:E_2}
\fint_{B_r} |\phi(x, |\nabla w|) - \tilde \phi(|\nabla w)| \chi_{E_2} \, dx \lesssim (r + \omega(2r)) \left (\tilde \phi\left( \fint_{B_{2r}} |\nabla u| \, dx \right ) + 1 \right ).
\end{align}

For $E_3$ we have the inequality $1< |B_{2r}|^{1-\varepsilon_0}\phi^-(|\nabla w|)$, which implies with Lemma~\ref{lem:tilde-phi} (3)
\begin{align*}
|\phi(x, |\nabla w|) - \tilde \phi(|\nabla w|) | &\lesssim \phi(x,|\nabla w|) + 1 \lesssim \phi(x,|\nabla w|)\\
&\leq [|B_{2r}|^{1-\varepsilon_0} \phi^-(|\nabla w|)]^{\sigma_0/2}\phi(x, |\nabla w|)\\
&\lesssim r^{n(1-\varepsilon_0)\sigma_0/2} \phi(x,|\nabla w|)^{1+\frac{\sigma_0}2}.
\end{align*}
Letting
$$ \varepsilon_0 : = \frac{\sigma_0 }{2(2+\sigma_0)}$$
and
integrating both sides over the set $E_3$ we obtain
\begin{align*}
\fint_{B_r}|\phi(x, |\nabla w|) - \tilde \phi(|\nabla w|) | \chi_{E_3} \, dx \lesssim r^{\frac{n(4+\sigma_0)\sigma_0}{4(2+\sigma_0)}} \left ( \fint_{B_r} \phi(x,|\nabla w|)^{1+\frac{\sigma_0}2} \, dx\right )^{\frac{2}{2+\sigma_0}+ \frac{\sigma_0}{2+\sigma_0}}.
\end{align*}

Firstly, \eqref{est:CZwu} and \eqref{est:reverse} imply that 
\begin{align*}
\left (\fint_{B_r} \phi(x,|\nabla w|)^{1+\frac{\sigma_0}2}\, dx \right )^{\frac2{2+\sigma_0}} \lesssim \tilde \phi\left (\fint_{B_{2r}} |\nabla u| \, dx \right ) +1.
\end{align*}
Secondly, by \eqref{eq:highernormsmall}
\begin{align*}
\left (\fint_{B_r} \phi(x, |\nabla w|)^{1+\frac{\sigma_0}2} \, dx \right )^{\frac{\sigma_0}{2+\sigma_0}} \leq |B_r|^{-\frac{\sigma_0}{2+\sigma_0}} \lesssim r^{-\frac{n\sigma_0}{2+\sigma_0}}.
\end{align*}
Using three previous estimates we find that
\begin{align}
\label{est:E_3}
\fint_{B_r}|\phi(x, |\nabla w|) - \tilde \phi(|\nabla w|) | \chi_{E_3} \, dx \lesssim r^{\frac{n\sigma_0^2}{4(2+\sigma_0)}} \left (\tilde \phi \left (\fint_{B_{2r}} |\nabla u| \, dx \right ) +1 \right ).
\end{align}
Combining \eqref{est:E_1}, \eqref{est:E_2} and \eqref{est:E_3} together, we have arrived at
\begin{align*}
|I_2| &\leq \fint_{B_r}|\phi(x, |\nabla w|) - \tilde \phi(|\nabla w|)| \, dx \\
& \lesssim   (\omega(2r)^{p/q} + r+ r^{\frac{n\sigma_0^2}{4(2+\sigma_0)}}) \left (\tilde \phi \left (\fint_{B_{2r}} |\nabla u| \, dx \right ) +1 \right )\\
& \lesssim   (\omega(2r)^{p/q} +r^{\gamma}) \left (\tilde \phi \left (\fint_{B_{2 r}} |\nabla u| \, dx \right ) +1 \right )
\end{align*}
where $\gamma= \min \{ 1, \frac{n\sigma_0^2}{4(2+\sigma_0)}\}$.
 The estimate for $I_1$ is analogous to $I_2$ with $\nabla u$ instead of $\nabla w$.

 We continue with estimating $I_3$.   We start by calculating
\begin{align*}
\partial \tilde \phi(|x|) - \partial \tilde \phi(y) &= \int_{0}^1 \dfrac{d}{dt} \partial \tilde \phi(tx + (1-t)y) \, dt \leq |x-y| \int_{0}^1 \tilde \phi''(|x|+|y|) \, dt \\
&\approx |x-y| \dfrac{\tilde \phi'(|x|+|y|)}{|x|+|y|}.
\end{align*}
Identifying $x = \nabla \psi$ and $y = \nabla \psi(x_0)$, where $x_0$ is the center of $B_r$, and using Lemma \ref{est:Orlicz} (2) we have
\begin{align*}
|I_3|&=\left| \fint_{B_r} \frac{\tilde{\phi}' (|\nabla \psi|)}{|\nabla \psi |} \nabla \psi  \cdot (\nabla w-\nabla u) \,dx\right|\\
& =\left| \fint_{B_r} \left[\frac{\tilde{\phi}' (|\nabla \psi|)}{|\nabla \psi |} \nabla \psi -\frac{\tilde{\phi}' (|\nabla \psi(x_0)|)}{|\nabla \psi(x_0) |} \nabla \psi (x_0)\right] \cdot (\nabla w-\nabla u) \,dx\right|\\
&\lesssim  \fint_{B_r} \frac{\tilde{\phi}' (|\nabla \psi|+|\nabla \psi(x_0)|)}{|\nabla \psi|+|\nabla \psi(x_0)|}  |\nabla \psi- \nabla \psi (x_0)| \left| \nabla w-\nabla u \right| \,dx.
\end{align*}
Since $\tilde{\phi}'$ satisfies \inc{p-1}, we note that if $p\geq 2,$ then  $\tilde{\phi}'(t)/ t$ is non-decreasing for $t \in (0,\infty).$
Then 
\begin{align*}
|I_3|&\lesssim\fint_{B_r} \frac{\tilde{\phi}' (|\nabla \psi|+|\nabla \psi(x_0)|)}{|\nabla \psi|+|\nabla \psi(x_0)|}  |\nabla \psi- \nabla \psi (x_0)| \left| \nabla w-\nabla u \right| \,dx
\\
&\lesssim r^{\beta}  \fint_{B_r} | \nabla u|+|\nabla w| \,dx \lesssim  r^{\beta}\left(  \fint_{B_r}  \tilde{\phi}(|\nabla u|) \,dx +1\right).
\end{align*}
If $p < 2,$ we deduce that 
\begin{align*}
|I_3|&\lesssim\fint_{B_r} \frac{\tilde{\phi}' (|\nabla \psi|+|\nabla \psi(x_0)|)}{|\nabla \psi|+|\nabla \psi(x_0)|}  |\nabla \psi- \nabla \psi (x_0)| \left| \nabla w-\nabla u \right| \,dx
\\
&\lesssim\fint_{B_r} \frac{\tilde{\phi}' (|\nabla \psi|+|\nabla \psi(x_0)|)}{|\nabla \psi|+|\nabla \psi(x_0)|} (|\nabla \psi|+|\nabla \psi(x_0)|)^{2-p} |\nabla \psi- \nabla \psi (x_0)|^{p-1} \left| \nabla w-\nabla u \right| \,dx
\\
&=\fint_{B_r} \frac{\tilde{\phi}' (|\nabla \psi|+|\nabla \psi(x_0)|)}{(|\nabla \psi|+|\nabla \psi(x_0)|)^{p-1}}  |\nabla \psi- \nabla \psi (x_0)|^{p-1} \left| \nabla w-\nabla u \right| \,dx
\\
&\lesssim r^{\beta(p-1)}  \fint_{B_r} | \nabla u|+|\nabla w| \,dx \lesssim  r^{\beta(p-1)}\left(  \fint_{B_r}  \tilde{\phi}(|\nabla u|) \,dx +1\right),
\end{align*}
since $\tilde{\phi}'(t)/ t^{p-1}$ is non-decreasing for $t \in (0,\infty)$.
In turn, we conclude that 
\begin{align*}
|I_3|\lesssim  r^{\beta}\left(  \fint_{B_r}  \tilde{\phi}(|\nabla u|) \,dx +1\right),
\end{align*}
where the implicit constant depends on $n, p,q,L,[\nabla \psi]_{\beta}, \overline{\nabla \psi}.$

Note from \eqref{est:reverse} and Lemma \ref{lem:tilde-phi} (3) that
\begin{align*}
\fint_{B_r}  \tilde{\phi}(|\nabla u|) \,dx &\lesssim \fint_{B_r}  \phi(x,|\nabla u|)  \,dx +1\\
 & \leq \left( \fint_{B_{r}}  \phi(x,|\nabla u|)^{1+\sigma_0}  \,dx\right)^{\frac1{1+\sigma_0}} +1 \lesssim \tilde{\phi} \left( \fint_{B_{2r}} |\nabla u|\,dx \right) +1.
\end{align*} 
Hence, for any $\kappa \in (0,1)$, we conclude that 
\begin{align*}
&\fint_{B_r} \tilde{\phi}( | \nabla u - \nabla w|) \,dx  \\
 &\lesssim \kappa \left(  \fint_{B_r } \tilde{\phi}( |\nabla u  |)  \,dx  + 1 \right) +  \kappa^{-1} (\omega(2r)+r^{\gamma} +r^{\beta}) \left[  \tilde{\phi}\left(\fint_{B_{2r}} |\nabla u |\,dx\right)+ 1\right]\\
    &\lesssim \left(\kappa +  \kappa^{-1}  [\omega(2r)^{p/q} +r^{\gamma}+ r^{\beta}]\right) \left[  \tilde{\phi}\left(\fint_{B_{2 r}} |\nabla u |\,dx\right)+ 1\right].
       \end{align*} 
 By taking $\kappa:=\left (\omega(2r)^{p/q} +r^{\gamma}+ r^{\beta}\right)^{\frac12},$  we then conclude that 
 $$ \fint_{B_r} \tilde{\phi}(|\nabla u - \nabla w|) \,dx \lesssim   \left (\omega(2r)^{p/q} + r^{\gamma}+ r^{\beta}\right)^{\frac12} \left[  \tilde{\phi}\left(\fint_{B_{2r}} |\nabla u |\,dx\right)+ 1\right].$$
In turn, by Jensen's inequality and \adec{q} of $\tilde{\phi}$, we obtain that
  \begin{align*}
\tilde{\phi}\left( \fint_{B_r} |\nabla u - \nabla w| \,dx \right) &\lesssim \fint_{B_r} \tilde{\phi}(|\nabla u - \nabla w|) \,dx\\
& \lesssim      \tilde{\phi} \left( (\omega(2r)^{p/q} + r^{\gamma}+ r^{\beta})^{\frac1{2q}}
\left[\fint_{B_{2r}} |\nabla u |\,dx+ 1\right]\right),
   \end{align*}
which implies the claim since $\tilde{\phi}$ is strictly increasing.
\end{proof}

With similar arguments we get a comparison estimate between $w$ and $v$.

\begin{lem}\label{lem:comparison-wv}
Suppose that $w$ is a solution to the equation \eqref{eq:wpsi}, $v$ is a solution to the equation \eqref{eq:vw} and $B_{2r} \Subset \Omega$. Then
 $$ \fint_{B_r} |\nabla w - \nabla v| \,dx  \leq c\,     r^{\frac{\beta}{2q}} \left(\fint_{B_{2r}} |\nabla u |\,dx+ 1\right)$$
for some constant $c = c(n,p,q,L,[\nabla \psi]_{\beta}, \overline{\nabla \psi})>0.$
\end{lem}

\begin{proof}
Since $w-v \in W_0^{1,\tilde{\phi}}(B_r)$, we infer from \eqref{eq:wpsi} and \eqref{eq:vw} that 
 \begin{align*}
&\int_{B_r} \left[ \frac{\tilde{\phi}'(|\nabla w|)}{|\nabla w|} \nabla w - \frac{\tilde{\phi}'(|\nabla v|)}{|\nabla v|} \nabla v \right] \cdot (\nabla w - \nabla v)\,dx \\
&=\int_{B_r} \left[ \frac{\tilde{\phi}'(|\nabla \psi|)}{|\nabla \psi|} \nabla \psi - \frac{\tilde{\phi}'(|\nabla \psi(x_0)|)}{|\nabla \psi(x_0)|} \nabla \psi(x_0) \right] \cdot (\nabla w - \nabla v)\,dx,
 \end{align*}
where $x_0$ is the center of $B_r.$ 
In a similar way as in the estimation of $I_3$ in the previous lemma, we see that 
 \begin{align*}
 &\fint_{B_r} \left[ \frac{\tilde{\phi}'(|\nabla \psi|)}{|\nabla \psi|} \nabla \psi - \frac{\tilde{\phi}'(|\nabla \psi(x_0)|)}{|\nabla \psi(x_0)|} \nabla \psi(x_0) \right] \cdot (\nabla w - \nabla v)\,dx\\
&\lesssim  r^{\beta}\left(  \fint_{B_r}  \tilde{\phi}(|\nabla u|) \,dx +1\right).
 \end{align*}
Therefore, by Lemma \ref{est:Orlicz} (2) \& (4) with \eqref{est:energy} and \eqref{est:energy2}, we obtain that
 \begin{align*}
& \fint_{B_r}  \tilde{\phi}(|\nabla w-\nabla v|)\,dx\\
 & \lesssim \kappa \left( \fint_{B_r}  \tilde{\phi}(|\nabla u |)\,dx+1\right) 
 + \kappa^{-1} r^{\beta}\left(  \fint_{B_r}  \tilde{\phi}(|\nabla u|) \,dx +1\right),
 \end{align*}
 which implies that $$ \fint_{B_r} \tilde{\phi}(|\nabla w - \nabla v|) \,dx \lesssim r^{\frac{\beta}{2}}  \left[  \tilde{\phi}\left(\fint_{B_{2r}} |\nabla u |\,dx\right)+ 1\right]$$
 by taking $\kappa:=r^{\frac{\beta}{2}}.$
 With Jensen's inequality, \adec{q} and increasingness of $\tilde \phi$, the claim follows in the same way as in the proof of the previous lemma.
\end{proof}

The following is the well-known technical iteration lemma (cf.  \cite[Lemma~7.1]{HasO_pp18} and \cite[Lemma~2.1, Chapter 3]{Giaquinta}).

\begin{lem}\label{lem:tech}
Let $g: [0,r_0] \rightarrow [0,\infty)$ be a non-decreasing function. Suppose that 
$$ g(\rho)  \leq C \left[ \left(\frac{\rho}{r}\right)^n +\epsilon \right]g(r) + C r^n$$
for all $0<\rho<r\le r_0$ with non-negative constant $C.$ Then for any $\mu\in(0,n)$ there exist $\epsilon_1=\epsilon_1(n,C,\mu)>0$ such that if $\epsilon < \epsilon_1$, we have 
$$ g(\rho)  \leq c \left(\frac{\rho}{r}\right)^{n-\mu} \left(  g(r)+r^{n-\mu} \right), $$
where $c$ is  a constant depending on $n, C$ and $\mu$.

\end{lem}

Now we prove our main results. In the results, $\omega$ is the function from \wvaone{} for $\epsilon=\epsilon_0$ and $L \ge 1$ is the constant from  \azero{}.  
We also remark that \wvaone{} can be replaced by the conditions \aone{} and \wvaone{} with fixed $\epsilon>0$ which is sufficiently small depending on $n, p, q, L$. 



\begin{proof}[Proof of Theorem~\ref{thm:C^1,alpha} (i)]

Let $r_0 \in (0,1)$ be a sufficiently small number which will be determined later. Consider any $B_{r_0} \subset \Omega' \Subset \Omega$ assuming that $r_0>0$ satisfies \eqref{ass:rsmall} with $r=r_0$,
and let $0<2r\le r_0.$

First we recall Jensen's inequality and \eqref{est:energy2} followed by Lemma \ref{lem:tilde-phi} (3) and \eqref{est:reverse} to obtain

\begin{align*}
\tilde \phi \left (\fint_{B_r} |\nabla v| \, dx \right ) &\leq \fint_{B_r} \tilde \phi(|\nabla v|) \, dx \lesssim \fint_{B_r} \tilde \phi(|\nabla u|)  \, dx+1 \\
& \lesssim  \left (\fint_{B_r} \phi(x, |\nabla u|)^{1+ \sigma_0} \, dx\right )^{\frac{1}{1+\sigma_0 }} +1\lesssim \tilde \phi \left (\fint_{B_{2r}} |\nabla u| \, dx\right ) +1.
\end{align*}
Now if we apply $\tilde \phi^{-1}$ to both sides of the inequality and use \adec{} to insert the $``+1"$ inside the $\Phi$-function we get
\begin{align}\label{eq:v-u-energy}
\fint_{B_r} |\nabla v| \, dx \lesssim \fint_{B_{2r}} |\nabla u| +1 \, dx.
\end{align}

For simplicity, we write $\omega_0(r):= (\omega(2r)^{p/q}+r^{\gamma} +r^{\beta})^{\frac1{2q}}.$
 Then if $\rho< \frac{r}{2},$ by Lemmas~\ref{lem:comparison-uw},\ref{lem:comparison-wv} and \eqref{v-Lip},
we have that 
 \begin{align*}
  \int_{B_{\rho}} |\nabla u |\,dx &\le  \int_{B_{r}} |\nabla u-\nabla w |\,dx +  \int_{B_{r}} |\nabla w-\nabla v|\,dx+ \int_{B_{\rho}}|\nabla v |\,dx\\
  & \lesssim\omega_0(r) \int_{B_{2r}} |\nabla u | +1\,dx +  \rho^n \sup_{B_{r/2}} |\nabla v |\\
  & \lesssim \omega_0(r_0)\int_{B_{2r}} |\nabla u | +1\,dx +  \rho^n \fint_{B_{r}} |\nabla v |\,dx \\
  &\lesssim \left( \omega_0(r_0)+\left(\frac{\rho}{r}\right)^n \right)\int_{B_{2r}} |\nabla u |\,dx + r^n.
    \end{align*}
Otherwise, i.e. if $\frac{r}{2}\leq \rho< 2r,$ the above estimate is clear because $ \frac12 \leq \frac{\rho}{r}.$

Therefore, by choosing $r_0$ so small that 
$\omega_0(r_0) = (\omega(r_0)+r_0^{\gamma} +r_0^{\beta})^{\frac1{2q}} < \epsilon_1,$
where $\epsilon_1$ is given in Lemma~\ref{lem:tech}, we have from Lemma~\ref{lem:tech} that for any $\mu\in(0,n),$ 
 \begin{align}\label{DuL1bdd}
 \int_{B_{\rho}} |\nabla u |\,dx  \lesssim\left(\frac{\rho}{r_0} \right)^{n-\mu} \left(  \int_{B_{r_0}} |\nabla u |\,dx  +r_0^{n-\mu}\right) 
    \end{align}
    for all $B_{\rho} \subset \Omega'$ with $\rho \in (0,r_0).$ In addition, $B_{\rho} \subset \Omega'$ is arbitrary and the implicit constant is universal.  Hence, we take  $1-\mu=\alpha$ to obtain that $u \in C_{\loc}^{0,\alpha}(\Omega')$  by Morrey type embedding.   
\end{proof}

%

\begin{proof}[Proof of Theorem~\ref{thm:C^1,alpha} (ii)]
Fix $\Omega' \Subset \Omega$. 
Recall from \eqref{nablavtau} that for any $0<\rho<\frac{r}{2}$
 $$ \fint_{B_{\rho}}  |\nabla v - (\nabla v)_{B_{\rho}} | \,dx \leq c \left(\frac{\rho}{r}\right)^{\alpha_0} \fint_{B_{r/2}} |\nabla v |\,dx$$
for some $\alpha_0 >0$.
From \eqref{DuL1bdd}, we note that for $\mu \in (0,1),$
$$ \fint_{B_{2r}} |\nabla u| \,dx \leq c_\mu r^{-\mu}$$
for all $B_{2r} \subset \Omega'$ with $r \in (0,\frac{r_0}{2}]$, where $r_0$ is from the proof of Theorem~\ref{thm:C^1,alpha} (i) and  the constant $c_\mu \ge 1$ depends on $n, p, q, L, r_0,$ and $\mu$.  
 Let us consider a small $r <\frac{r_0}{2}$ which will be determined later. 
From Lemmas~\ref{lem:comparison-uw}, \ref{lem:comparison-wv} and \eqref{eq:v-u-energy},
we then derive that  for $0<\rho<\frac{r}{2}$,
  \begin{align*}
 & \fint_{B_{\rho}} |\nabla u - (\nabla u)_{B_{\rho}} | \,dx \leq 2 \fint_{B_{\rho}} |\nabla u - (\nabla v)_{B_{\rho}} | \,dx \\
  &\quad   \leq 2 \fint_{B_{\rho}} |\nabla u - \nabla w | \,dx  +2 \fint_{B_{\rho}} |\nabla w - \nabla v | \,dx +   2\fint_{B_{\rho}} |\nabla v - (\nabla v)_{B_{\rho}} | \,dx\\
  &\quad   \lesssim \left(\frac{r}{\rho}\right)^n  \fint_{B_{r}} |\nabla u - \nabla w | \,dx  + \left(\frac{r}{\rho}\right)^n \fint_{B_{r}} |\nabla w - \nabla v | \,dx +    \left(\frac{\rho}{r}\right)^{\alpha_0}  \fint_{B_{r/2}} |\nabla v |\,dx\\
    &\quad   \lesssim    \left(\frac{r}{\rho}\right)^n (\omega(2r)^{p/q}+r^{\gamma} +r^{\beta})^{\frac1{2q}} \left(\fint_{B_{2r}} |\nabla u |\,dx+ 1\right) +    \left(\frac{\rho}{r}\right)^{\alpha_0}  \fint_{B_{r/2}} |\nabla v |\,dx\\
&\quad   \lesssim \left( r^{\delta_0} \left(\frac{r}{\rho}\right)^n +  \left(\frac{\rho}{r}\right)^{\alpha_0}\right)  \left(\fint_{B_{2r}} |\nabla u |\,dx+ 1\right)  \leq c_{\mu}  r^{-\mu} \left( r^{\delta_0} \left(\frac{r}{\rho}\right)^n +  \left(\frac{\rho}{r}\right)^{\alpha_0}\right),
    \end{align*}
    where $\delta_0 : = \frac{1}{2q}\min\{ \frac{\delta p}{q}, \gamma, \beta \}.$
   In the same way as in the proof of Theorem~7.4 in \cite{HasO_pp18}, we obtain $\nabla u \in C_{\loc}^{\alpha}(\Omega')$ for some $\alpha>0.$ 
       \end{proof}

\bigskip

\noindent\small{
\textsc{A. Karppinen}}\\
\small{Department of Mathematics and Statistics,
FI-20014 University of Turku, Finland}\\
\footnotesize{\texttt{arttu.a.karppinen@utu.fi}}\\
\noindent\small{
\textsc{M.\ Lee}}\\
\small{Department of Mathematics, Pusan National University, Busan 46241, Republic of Korea}\\
\footnotesize{\texttt{mikyounglee@pusan.ac.kr}}
\end{document}